\documentclass[a4paper, english, 10pt]{amsart}

\usepackage{amsthm,amssymb,url,hyperref}
\usepackage{fullpage}
\usepackage{stmaryrd}
\usepackage[all]{xy}
\usepackage{color}
\usepackage[all]{xy}
\usepackage{tikz}
\usepackage{color,tikz}

\theoremstyle{plain}
\newtheorem{theorem}{Theorem}[section]

\newtheorem{question}[theorem]{Question}

\newtheorem{corollary}[theorem]{Corollary}

\newtheorem{lemma}[theorem]{Lemma}

\newtheorem{proposition}[theorem]{Proposition}

\theoremstyle{definition}

\newtheorem{problem}[theorem]{Problem}

\newtheorem{remark}[theorem]{Remark}

\newtheorem*{remark*}{Remark}

\def\!#1{#1^\s}
\def\ld{\backslash}

\def\ker#1{\mathrm{ker}(#1)}

\def\aut#1{\mathrm{Aut}(#1)}

\def\End#1{\mathrm{End}(#1)}
\def\aff#1{\mathrm{Aff}#1}

\def\lmlt{\mathrm{LMlt}}
\def\dis{\mathrm{Dis}}
\def\sym{\mathrm{Sym}}
\def\Sym{\mathrm{Sym}}

\def\s{\mathfrak{s}}

\def\comment#1{{\color{red} #1}}

\def\setof#1#2{\{#1\, : \,#2\}}

\def\m{\mathfrak{ip}}

\def\O{\mathcal O}
\def\inv{^{-1}}
\def\ldiv{\backslash}

\def\N{\mathrm{Norm}}
\def\c#1{\mathrm{con}_{#1}}

\def\cg#1{\equiv_\alpha}
\newcommand*\xbar[1]{%
   \hbox{%
     \vbox{%
       \hrule height 0.5pt 
       \kern0.5ex
       \hbox{%
         \kern-0.1em
         \ensuremath{#1}%
         \kern-0.1em
       }%
     }%
   }%
}

\usepackage{MnSymbol}

\title{Medial and semimedial left quasigroups}
\author{Marco Bonatto}
\address{Instituto IMAS - CONICET, Buenos Aires}
\email{marco.bonatto.87@gmail.com}

\begin{document}

\maketitle

\section*{Abstract}

In this paper we investigate the class of semimedial left quasigroups, a class that properly contains racks and medial left quasigroups. We extend most of the results about commutator theory for racks collected in \cite{CP} and some of the results concerning Malt'sev conditions for quandles collected in \cite{Maltsev_paper} to the class of semimedial left quasigroups.

\section*{Introduction}

Left quasigroups are rather combinatorial objects, but often binary algebraic structures of interests have an underlying left quasigroup structure: examples are {\it racks and quandles}, motivated by knot theory \cite{AG, J}, or other structures, arising in the study of the solution of the set-theoretic Yang-Baxter equation \cite{Rump, Rumples}.

One goal of this paper is to provide some general tools for the study of the underlying left quasigroups of such structures, taking inspiration on the recent developments in racks and quandles theory, in particular towards the interplay between congruences and the displacement group on one hand, and the commutator theory in the sense of Freese-McKenzie \cite{CP} on the other.  

A natural setting for the commutator theory is the class of {\it left translation} (LT) left quasigroups, already introduced in \cite{CP}: indeed, for such left quasigroups the notion of {\it abelianness} and {\it centrality} of congruence (in the sense of \cite{comm}) is completely captured by of group theoretical properties of the {\it relative displacement groups}.

{\it Semimedial} left quasigroups, axiomatized by the identity
\begin{equation}
(x*x)*(y*z)\approx (x*y)*(x*z),\tag{SM1}
\end{equation}
have the LT property. The semimedial identity has been also studied in the framework of quasigroups as in \cite{book_quasi} and \cite{MJD}. This class properly contains both {\it racks}, i.e. the left quasigroups satisfying the identity 
\begin{equation}\label{LD}
 x*(y*z)\approx (x*y)*(x*z),\tag{LD}
 \end{equation}
  and {\it medial left quasigroups}, obeying the law
\begin{equation}\label{M}
 (x*y)*(z*u)\approx (x*z)*(y*u).\tag{M}
\end{equation}
Racks and {\it quandles} (i.e. idempotent racks) are the inspiring and leading example for the present paper. For racks the interplay between commutator theory and the properties of the displacement group is quite strong. Indeed, also the notion of {\it solvability} and {\it nilpotence} is reflected by the correspondent properties of the displacement group \cite{CP}. 

The variety of semimedial left quasigroups turned out to be a natural framework to extend the theory already developed for racks. Indeed, most of the results of \cite{CP} can be extended to semimedial left quasigroups and most of the argument and ideas still go through with very little adjustments.

The paper is organized as follows: in Section \ref{Sec 1} we recall all the basic definitions in the framework of left quasigroups. In Theorem \ref{Galois for orbits} we also show that there is a non-trivial interplay between congruences and displacement group of left quasigroups as we can define a Galois connection between the congruence lattice and a sublattice of certain normal subgroup of the left multiplication group that we call {\it admissible subgroups} (such connection was already noticed for racks in \cite{CP}).

In Section \ref{Sec 2} we recall the basics of commutator theory and summarize the results obtained for LT left-quasigroups. In Proposition \ref{dis solv then Q solvable for LT} we partially extend the correspondence between solvability (resp. nilpotence) of a LT left quasigroup and its displacement group, namely we prove that if the displacement group admits an abelian (resp. central) series of {\it admissible subgroups} then also the left quasigroup is solvable (resp. nilpotent).

Section \ref{Sec 3} is dedicated to semimedial left quasigroups. We show that several properties of racks can be extended to semimedial left quasigroups. We prove that solvability and nilpotence are completely reflected by the corresponding properties of the displacement groups (as for racks) in Proposition \ref{solvable then dis solv}. In Theorem \ref{Galois for semimedial} we prove that we have a second Galois connection between the congruence lattice and the lattice of the admissible subgroups, as for racks. Moreover, semimedial left quasigroups have the {\it Cayley property} and the {\it strongly solvable} ones (in the sense of \cite{TCT}) are nilpotent, see Corollary \ref{multip}.  

In the rest of the section we study two subclasses of semimedial subgroups, defined according to the property of the {\it squaring mapping}: {\it multipotent} and {\it $2$-divisible} left quasigroups (note that racks are {\it $2$-divisible} semimedial left quasigroups). In Theorem \ref{equivalence} we prove that there exists an isomorphism of categories between the category of $2$-divisible left quasigroups and a suitable category which objects are given by pairs $(Q,f)$ where $Q$ is a quandle and $f$ is an automorphism of $Q$. This isomorphism preserves the displacement group and accordingly nilpotence and solvability. In particular, we prove that finite $2$-divisible semimedial quasigroups are solvable in Corollary \ref{semimedial quasi are solv}.

In Section \ref{Sec 4} we turn our attention to medial left quasigorups.  We prove that medial left quasigroups are nilpotent of length at most $2$ in Corollary \ref{medial 2 div are nilp}. Using the isomorphisms of categories in Theorem \ref{equivalence} we can use the structure theory for medial quandles in \cite{Medial} in order to have a complete description also for the structure of $2$-divisible medial left quasigroups. In Theorem \ref{finite medial connected racks} we provide a classification of finite connected medial racks 

In Section \ref{Sec 5} we investigate Malt'sev varieties of semimedial left quasigroups and we can partially extend the results about Malt'sev varieties of quandles obtained in \cite{Maltsev_paper}.

Finally in Section \ref{Sec 6}, we introduce the class of {\it spelling left quasigroup} as a subclass of the LT left quasigroups containing racks and we prove that several features of racks can be extended to this class.

\subsection*{Acknowledgments}

The author wants to thanks Michael K. Kinyon for pointing him out the relevant identity defining the semimedial law.

\section{Left quasigroups}\label{Sec 1}


A {\it left quasigroup} is a binary algebraic structure $(Q,*,\ldiv)$ such that the following identities hold:
\begin{displaymath}
x*(x\ldiv y)\approx y\approx x\ldiv (x*y).
\end{displaymath}

Hence, a left quasigroup is a set $Q$ endowed with a binary operation $*$ such that the mapping $L_x:y\mapsto x*y$ is a bijection of $Q$ for every $x\in Q$. Clearly the left division is defined by $x\ldiv y=L_x^{-1}(y)$, so we usually denote left quasigroups just as a pair $(Q,*)$. A {\it term} in the language of left quasigroup is any meaningful nesting of the basic operations $\{*,\ldiv\}$: e.g. $t(x,y,z)=x*(y\ldiv (x*(z*z))$. 

A {\it quasigroup} is a binary algebraic structure $(Q,*,\ldiv,/)$ such that $(Q,*,\ldiv)$ is a left quasigroup and $(Q,*,/)$ is a right quasigroup defined analogously. If $(Q,*,\ldiv,/)$ is finite then it is {\it term equivalent} to its left quasigroup reduct $(Q,*,\ldiv)$ (right divisions can be expressed as a power of right multiplications).

%
%
We will denote the classes of an equivalence relation $\alpha$ over a set $Q$ as $[a]_\alpha$ and the correspondent quotient set as $Q/\alpha$. A {\it congruence} of a left quasigroup $Q$ is an equivalence relation $\alpha$ such that the following implication holds:
\begin{equation}\label{congruence prop}
a\, \alpha\, c \text{ and } b\, \alpha\, d \quad\Rightarrow \quad L_a^{\pm 1}(b)\, \alpha \, L_c^{\pm 1}(d) 
\end{equation}
Note that to check \eqref{congruence prop} it is enough to have that $L_a^{\pm 1}(c)\, \alpha \, L_b^{\pm 1}(c)$, $L_c^{\pm 1}(a)\, \alpha \, L_c^{\pm 1}(b)$ for every $a\, \alpha\, b$ and every $c\in Q$. Indeed, if $a\, \alpha\, c$ and $b\, \alpha\, d$ then
$$L_a^{\pm 1}(b)\, \alpha\, L_c^{\pm 1}(b) \, \alpha\, L_c^{\pm 1}(d).$$
Congruences of a left quasigroups $Q$ form a lattice denoted by $Con(Q)$ and they are in one-to-one correspondence with homomorphic images. Indeed a mapping $f:(Q,*)\longrightarrow (R,*)$ is a homomorphism of left quasigroups if and only if the equivalence relation $\ker{f}=\setof{(x,y)\in Q^2}{f(x)=f(y)}$ is a congruence. On the other hand, if $\alpha\in Con(Q)$ then the operations
$$[a]_\alpha * [b]_\alpha=[a*b]_\alpha, \quad [a]_\alpha \ldiv [b]_\alpha=[a\ldiv b]_\alpha,$$
are well defined and they provide a left quasigroup structure on $Q/\alpha$ and	the natural map $a\mapsto [a]_\alpha$ is a homomorphism. The congruence lattice of $Q/\alpha$ is given by $$Con(Q/\alpha)=\setof{\beta/\alpha}{\alpha\leq \beta\in Con(Q)}$$ where $[a]_\alpha \, \beta/\alpha\, [b]_\alpha$ if and only if $a\, \beta\, b$ (see \cite{UA} for further details). A congruence is {\it uniform} if all its blocks have the same size. 

A left quasigroup $Q$ is {\it connected} if the {\it left multiplication group} $\lmlt(Q)=\langle L_a,\, a\in Q\rangle$ is transitive on $Q$. If all the subalgebras of $Q$ are connected, then $Q$ is said to be {\it superconnected}. In particular congruences with connected factor are uniform \cite[Proposition 2.5]{CP}.
%
%
%
%
%
%

We denote by $Sg(S)$ the {\it subalgebra generated by a subset $S\subseteq Q$}, i.e. the smallest subalgebra of $Q$ containing $S$. In particular, $Sg(S)$ is given by all the left quasigroups terms evaluated on elements of $S$.

An element of $a\in Q$ is said to be {\it idempotent} if $a*a=a$. We denote by $E(Q)$ the set of the idempotent elements of $Q$ and we say that $Q$ is idempotent if $Q=E(Q)$.

Let $A$ be an abelian group, $f\in \End{A}$, $g\in \aut{A}$ and $c\in A$. The left quasigroup $(A,*)$ where
$$a*b=f(a)+g(b)+c$$
is denote by $\aff(A,f,g,c)$ and it is called {\it affine} left quasigroup over $A$.

\subsection*{The Cayley kernel}

The {\it Cayley kernel} of a left-quasigroup $Q$ is the equivalence relation $\lambda_Q$, defined by setting
\begin{displaymath}
a\, \lambda_Q\, b\, \text{ if and only if } \,L_a=L_b.
\end{displaymath}
If $\lambda_Q=0_Q$, we say that $Q$ is {\it faithful}. Note that if $Q/\alpha$ is faithful, then $\lambda_Q\leq \alpha$. 

\begin{lemma}\label{injective right}
Left quasigroups with injective right multiplications are faithful.
\end{lemma}

\begin{proof}
Assume that the right multiplication of $Q$ are injective and assume that $L_a=L_b$. Then $a*c=b*c$ for every $c\in Q$ and so $a=b$.
\end{proof}

For any set $S$ and any mapping $\theta:Q\times Q\longrightarrow \sym_S$ the algebraic structure $Q\times_\theta S=(Q\times S,*)$ where
\begin{equation}\label{extensions by theta}
(a,s)*(b,t)=(a*b,\theta_{a,b}(t))
\end{equation}
for every $a,b\in Q$ and every $s,t\in S$ is a left quasigroup and the canonical projection $Q\times_\theta S\longrightarrow Q$ is a homomorphism and its kernel is contained in $\lambda_Q$. On the other hand, if a left quasigroup $Q$ has a uniform congruence $\alpha\leq \lambda_Q$ with blocks of cardinality $|S|$, then $Q\cong Q/\alpha\times_\theta S$ for a suitable mapping $\theta:Q/\alpha\times Q/\alpha\longrightarrow \sym_S$ (the proof is the same as in \cite[Section 2.2, 2.3]{AG}). Moreover, if $\gamma:Q\longrightarrow \sym_S$ and
\begin{equation}\label{cohomologous}
\varepsilon_{a,b}=\gamma_{a*b}\theta_{a,b}\gamma_b^{-1}
\end{equation}
 for every $a,b\in Q$ then $Q\times_\theta S\cong Q\times_\varepsilon S$.

If $\lambda_Q=1_Q$, i.e. $a*b=\theta(b)$ for some $\theta\in \sym_Q$ for every $a,b\in Q$ we that that $Q$ is a {\it permutation left-quasigroup} and we denote it by $(Q,\theta)$. For instance the affine left quasigroup $\aff(C,0,1,1)$ where $C$ is a cyclic group and
$$a*b=b+1$$
for all $a,b\in C$, is a permutation left quasigroup (note that all the connected permutation left quasigroups are obtained in this way).  Idempotent permutation left quasigroup, satisfies $x*y\approx \approx y*y\approx y$ and they are called {\it projection}. The one element projection left quasigroup is called {\it trivial}.

In general $\lambda_Q$ is not a congruence of $Q$. If this is the case we say that $Q$ has the {\it Cayley property} or that $Q$ is a {\it Cayley left quasigroup}. 

Let $\mathcal{V}$ be a class of Cayley left-quasigroup closed under homomorphic images and $Q\in \mathcal{V}$. We define:
\begin{equation}\label{chain of L}
L_0(Q)=Q,\quad L_{n+1}(Q)=L_n(Q)/\lambda_{L_n(Q)},
\end{equation}
for every $n\in \mathbb{N}$. Following \cite{Modes} and \cite{covering_paper} we say that $Q$ is {\it $n$-reductive} if $Q$ satisfies the identity
\begin{equation}\label{reductive}
(\ldots((y*x_1)*x_2)*\ldots)*x_n \approx (\ldots((z*x_1)*x_2)*\ldots)*x_n,
\end{equation}
or equivalently if $|L_n(Q)|=1$. For the binary algebras arising from the solution to the Yang-Baxter equation, the algebraic structures satisfying \eqref{reductive} are also called {\it multipermutational} \cite{Gateva, Jedlicka}.

The identities \eqref{reductive} are related to the notion of {\it strongly abelianess}, as defined in \cite{TCT}. A congruence $\alpha$ of an algebraic structure $A$ is {\it strongly abelian} if the following implication holds
$$t(a,\bar{x})=t(b,\bar{y})\, \Rightarrow\, t(a,\bar{z})=t(b,\bar{z})$$
for every term $t$, every $a\, \alpha\, b$ and every tuple $\bar{z}$. An algebraic structure is {\it strongly solvable} if there exists a chain of congruences
$$0_A=\alpha_0\leq \alpha_1\leq \alpha_2\leq \ldots \leq \alpha_n=1_A$$
such that $\alpha_{i+1}/\alpha_i$ is strongly abelian in $A/\alpha_i$ for every $1\leq i \leq n-1$. According to \cite[Theorem 5.3]{covering_paper} we have that a Cayley left quasigroup $Q$ is $n$-strongly solvable if and only if $Q$ is $n$-reductive.

\subsection*{The squaring mapping}

We will denote the {\it squaring mapping} of a left quasigroup $Q$ as
$$\s:Q\longrightarrow Q,\quad a\mapsto a*a.$$
 We say that $Q$ is {\it (uniquely) $2$-divisible} if $\s$ is a bijection. Moreover, if $Q/\alpha$ is uniquely $2$-divisible then $\ker{\s}=\setof{(a,b)\in Q^2 }{\s(a)=\s(b)}\leq \alpha$.  Indeed, if $\s(a)=\s(b)$ then $[\s(a)]_\alpha=\s([a]_\alpha)=[\s(b)]_\alpha=\s([b]_\alpha)$ and therefore $[a]_\alpha=[b]_\alpha$.

For a left quasigroup $Q$ we also define
$$\s^0(Q)=Q,\quad \s^{n+1}(Q)=\s(\s^{n-1}(Q))$$
for every $n\in \mathbb{N}$. We say that $Q$ is {\it $n$-multipotent} if $|\s^n(Q)|=1$. If $n=1$ we say that $Q$ is {\it unipotent}.

\begin{lemma}\label{s and lambda}
Let $Q$ be a left quasigroup. Then $\ker{\s}\cap \lambda_Q=0_Q$. In particular, unipotent left quasigroups are faithful.
\end{lemma}

\begin{proof}
Assume that $a\, (\ker{\s}\cap \lambda_Q)\, b$. Then $a*a=b*b=a*b$ which implies $a=b$. If $Q$ is unipotent, then $\ker{\s}\cap \lambda_Q=\lambda_Q=0_Q$ and so $Q$ is faithful. 
\end{proof}

%
%
%
%
%
%

%
%

\begin{lemma}
Let $Q$ be a multipotent left quasigroup. Then $Q$ has a unique idempotent element and it is contained in every subalgebra of $Q$.
\end{lemma}

\begin{proof}
Let $\s^n(Q)=\{e\}$ Then $\s(e)=\s(\s^n(x))=\s^n(\s(x))=e$ and so $e\in E(Q)$. If $a\in E(Q)$ then $a=\s^n(a)=e$. Moreover, $e=s ^n(a)\in Sg(a)$ for every $a\in Q$ and so it is contained in every subalgebra of $Q$.
\end{proof}

\subsection*{Congruences and subgroups}

Let $\alpha$ be a equivalence relation on a left quasigroup $Q$. The group
$$\dis_{\alpha}=\langle \setof{h L_a L_b\inv h\inv}{ a\,\alpha\,b,\,h\in \lmlt(Q)}\rangle$$
is called the {\it displacement group relative to $\alpha$} and it is a normal subgroup of $\lmlt(Q)$ (indeed it is the normal closure of the set $\setof{L_a L_b^{-1} }{a\, \alpha\, b}$ in $\lmlt(Q)$). Such groups have been defined for congruences of racks in \cite{CP}. The following combinatorial characterization follows by the same argument given in in \cite[Lemma 3.1]{CP}:
$$\dis_\alpha=\setof{L_{a_1}^{k_1}\ldots L_{a_n}^{k_n}L_{b_n}^{-k_n}\ldots L_{b_1}^{-k_1}}{a_i\, \alpha\, b_i \, \text{ for every }1\leq i \leq n, \, n\in \mathbb{N}}.$$

The displacement group relative to $1_Q$ is called {\it the displacement group of $Q$} and denoted simply by $\dis(Q)$. According to \cite[Lemma 1.1]{Maltsev_paper}, it has the following characterization
\begin{equation}\label{disQ combinatorial}
\dis(Q)=\setof{L_{a_1}^{k_1},\ldots L_{a_n}^{k_n}}{\sum_{i=1}^n k_i=0},
\end{equation}
and in particular $\lmlt(Q)=\dis(Q)\langle L_a\rangle$ for every $a\in Q$. 

We define the {\it kernel relative to $\alpha$} as
\begin{equation}\label{kernel}
\dis^\alpha = \setof{h\in \dis(Q)}{h(a)\, \alpha\, a\, \text{ for every }a\in Q}.
\end{equation}
Clearly both the assignments $\alpha\mapsto \dis_\alpha$, $\alpha\mapsto \dis^\alpha$ are monotone.

If $\alpha$ is a congruence then the following mapping
\begin{equation}\label{pi_alpha}
\pi_\alpha:\lmlt(Q)\longrightarrow \lmlt(Q/\alpha),\quad L_a\mapsto L_{[a]_\alpha}
\end{equation}
can be extended to a well defined group homomorphism with kernel denoted by $\lmlt^\alpha$   \cite[Section 2.3]{CP}. Moreover \eqref{pi_alpha} restricts and corestricts to $\dis(Q)$ and $\lmlt^\alpha\cap \dis(Q)=\dis^\alpha$ as defined in \eqref{kernel} and clearly $\dis_\alpha\leq \dis^\alpha$.

\begin{lemma}\label{caratt congruences}
Let $Q$ be a left quasigroup and $\alpha$ an equivalence relation on $Q$. The following are equivalent:
\begin{itemize}
\item[(i)] $\alpha\in Con(Q)$.
\item[(ii)] the blocks of $\alpha$ are blocks with respect to the action of $\lmlt(Q)$ and $\dis_\alpha\leq \dis^\alpha$.
\end{itemize}
\end{lemma}

\begin{proof}
(i) $\Rightarrow$ (ii) If $\alpha$ is a congruence, $L_c^{\pm 1}(a)\, \alpha \, L_c^{\pm 1}(b)$ and so the blocks of $\alpha$ are blocks for the action of $\lmlt(Q)$. 

(ii) $\Rightarrow$ (i) The blocks of $\alpha$ are blocks with respect to the action of $\lmlt(Q)$ so the map
$$\pi_\alpha:\lmlt(Q)\longrightarrow \Sym_{Q/\alpha}	$$
where $\pi_\alpha(h)([a]_\alpha)=[h(a)]_\alpha$ for every $a\in Q$ is a well defined group homomorphism and $\dis^\alpha=\ker{\pi_\alpha}\cap \dis(Q)$. Let $a\, \alpha \, b$ and $c\, \alpha\, d$. Then, since $L_a L_b^{-1}\in \dis^\alpha$ we have
$$\pi_\alpha(L_a L_b^{-1})[c] =[a*(b\ldiv c)]=[c],$$
and so $(a\ldiv c)\, \alpha\, (b\ldiv d)$. Since also $L_a^{-1} L_b\in \dis^\alpha$ then $(a*c)\, \alpha\, (b*d)$.
\end{proof}

Let $N\leq \sym_Q$: we denote by $\mathcal{O}_N$ the orbit decomposition with respect to the action of $N$. Clearly, $N\leq \dis^{\mathcal{O}_N}$.  We can associate to $N$ another equivalence relation:
$$\c{N}=\setof{(a,b)\in Q^2}{L_a L_b^{-1}\in N}.$$
The relation $\c{N}$ is obviously reflexive and if $a\, \c{N}\, b\, \c{N}\, c$ then
\begin{eqnarray*}
(L_a L_b^{-1})^{-1}=L_b L_a^{-1}\in N,\quad  L_a L_b^{-1}L_b L_c^{-1}=L_a L_c^{-1}\in N,
\end{eqnarray*}
i.e. $\c{N}$ is symmetric and transitive.

The relation above can be used to characterize faithful factors of a left quasigroups (we extend \cite[Proposition 3.7]{CP} to arbitrary left quasigroups).

\begin{proposition}\label{p:con_dis}Let $Q$ be a left quasigroup and $\alpha$ its congruence. Then $Q/\alpha$ is faithful if and only if $\alpha=\c {\dis^{\alpha}}$.
\end{proposition}

\begin{proof}
Clearly $\alpha\leq \c{\dis_\alpha}\leq \c{\dis^{\alpha}}$ since $\dis_\alpha\leq \dis^\alpha$. We prove that the equivalence $\lambda_{Q/\alpha}$ is equal to  
$$\c{\dis^{\alpha}}/\alpha=\setof{([a],[b])\in ( Q/\alpha)^2}{a\,\c{\dis^\alpha}\, b}.$$ Indeed, $[a]\, (\c{\dis^{\alpha}}/\alpha) \, [b]$ in and only if $L_a L_b^{-1}\in \dis^{\alpha}$, i.e. $L_{[a]}=L_{[b]}$.
Therefore, $Q/\alpha$ is faithful if and only if $\c{\dis^\alpha }/\alpha=0_{Q/\alpha}$, that is, if and only if $\alpha=\c {\dis^{\alpha}}$.
\end{proof}

Note that $\mathcal{O}_N\leq \c{N}$ if and only if $\dis_{\mathcal{O}_N}\leq N$.
Let us denote by 
$$\N(Q)=\setof{N\trianglelefteq \lmlt(Q)}{\mathcal{O}_N\leq \c{N}}.$$ 
the set of {\it admissible subgroups} of $\dis(Q)$.
\begin{lemma}\label{N(Q)}
Let $Q$ be a left-quasigroup and $\alpha\in Con(Q)$. Then:
\begin{itemize}
\item[(i)] $\N(Q)$ is a sublattice of the lattice of normal subgroups of $\lmlt(Q)$.
\item[(ii)] If $H\in \N(Q)$ then $\pi_\alpha(H)\in \N(Q/\alpha)$.
\end{itemize}
\end{lemma}

\begin{proof}
(i) Let $N,M\in \N(Q)$ and $g\in N$ and $h\in M$. Then
$$L_{gh(a)}L_a^{-1}=\underbrace{L_{gh(a)}L_{h(a)}^{-1}}_{\in\dis_{\mathcal{O}_N}\leq N} \underbrace{L_{h(a)}L_a^{-1}}_{\in \dis_{\mathcal{O}_M}\leq M}$$
for every $a\in Q$. Then $NM\in \N(Q)$. Let $h\in N\cap M$ then $L_{h(a)} L_a^{-1}\in \dis_{\mathcal{O}_N}\cap \dis_{\mathcal{O}_M}\leq N\cap M$. Hence $N\cap M\in \N(Q)$. 

(ii) If $H\in \N(Q)$, then $\pi_\alpha(H)$ is normal in $\lmlt(Q/\alpha)$. Since $\dis_{\mathcal{O}_N}\leq H$ so we have that 
$$L_{\pi_\alpha(h)([a])}L_{[a]}^{-1}=L_{[h(a)]}L_{[a]}^{-1} =\pi_\alpha(L_{h(a)} L_{a}^{-1})\in \pi_\alpha(H)$$
for every $h\in H$. Hence $\dis_{\O_{\pi_\alpha(H)}}\leq \pi_\alpha(H)$. 
\end{proof}

\begin{lemma}\label{on orbits}
Let $Q$ be a left quasigroup, $N\leq \sym_{Q}$ such that $\dis_{\mathcal{O}_N}\leq N$ and $\lmlt(Q)$ normalizes $N$. Then $\mathcal{O}_N$ is a congruence of $Q$. 
\end{lemma}
\begin{proof}
 Since $\lmlt(Q)$ normalizes $N$ then the orbits of $N$ are blocks with respect to the action of $\lmlt(Q)$. Moreover $\dis_{\mathcal{O}_N}\leq N\leq \dis^{\mathcal{O}_N}$. So we can apply Lemma \ref{caratt congruences} and then $\mathcal{O}_N$ is a congruence.
%
\end{proof}

The following Corollary extends \cite[Theorem 6.1]{Even} and \cite[Lemma 2.6]{CP} to arbitrary left quasigroups.
\begin{corollary}\label{orbits of dis}
Let $Q$ be a left quasigroup, $N\in\N(Q)$ and $\alpha\in Con(Q)$. Then:
\begin{itemize}
\item[(i)] $\dis_\alpha$ and $\dis^\alpha\in \N(Q)$. 

\item[(ii)]  $\mathcal{O}_N$ is a congruence and  $\alpha\circ \mathcal{O}_N=\mathcal{O}_N\circ \alpha$.
\end{itemize}

\end{corollary}

\begin{proof}

(i) The groups $\dis_\alpha$ and $\dis^\alpha$ are normal in $\lmlt(Q)$ and $\mathcal{O}_{\dis_\alpha}\leq \mathcal{O}_{\dis^\alpha}\leq \alpha$ and therefore, using that the assignment $\beta\mapsto \dis_\beta$ is monotone, we have 
$$\dis_{\mathcal{O}_{\dis_\alpha}}\leq \dis_{\mathcal{O}_{\dis^\alpha}}\leq \dis_\alpha\leq \dis^\alpha.$$ 

(ii) According to Lemma \ref{on orbits}, $\mathcal{O}_N$ is a congruence. Let $a\, \alpha \, c\,  \mathcal{O}_N \, b$, i.e. there exists $h\in N$ such that $c=h(b)\, \alpha\, a$. The blocks of $\alpha$ are blocks with respect to the action of $N$, so $h([b]_\alpha)=[a]_\alpha$. Therefore, $a\, \mathcal{O}_N\, h^{-1}(a)\, \alpha\, b$ and so $\mathcal{O}_N\circ \alpha=\alpha\circ \mathcal{O}_N$.
\end{proof}

%
%

The orbit decomposition with respect to the left multiplication group of a left quasigroup is the smallest congruence with a projection factor \cite[Proposition 1.3]{Maltsev_paper}. The orbit decomposition with respect to the displacement group has a similar characterization.

\begin{corollary}\label{on orbits2}
Let $Q$ be a left quasigroup. Then $\mathcal{O}_{\dis(Q)}$ is the smallest congruence of $Q$ with permutation factor. 
\end{corollary}

\begin{proof}
The relation $\pi=\mathcal{O}_{\dis(Q)}$ is a congruence according to Corollary \ref{orbits of dis}. Since $\dis(Q)\leq \dis^{\pi}$ then $\dis(Q/\pi)=1$ and therefore $L_{[a]}=L_{[b]}$ for every $a,b\in Q$, i.e. $Q/\pi$ is a permutation left quasigroup. Assume that  $Q/\beta$ is a permutation left quasigroup. Then $\dis(Q/\beta)=1$. Hence $\dis(Q)\leq \dis^{\beta}$ and so $\pi \leq \beta$.  
\end{proof}

The Galois connection between $\N(Q)$ and $Con(Q)$ observed for racks in \cite[Remark 3.9]{CP} can be defined for arbitrary left quasigroups.

\begin{theorem}\label{Galois for orbits}
Let $Q$ be a left quasigroup. The assignments $N\mapsto \mathcal{O}_N$ and $\alpha\mapsto \dis^\alpha$ provide a monotone Galois connection between $\N(Q)$ and $Con(Q)$. 
\end{theorem}

\begin{proof}
By the definition of the kernel \eqref{kernel} we have that $N\leq \dis^\alpha$ if and only if $\mathcal{O}_N\leq \alpha$.
\end{proof}

For a rack $Q$ we have that $\N(Q)$ coincides with the lattice of normal subgroups of $\lmlt(Q)$.

\begin{lemma}
Let $Q$ be a left quasigroup and $N\leq \aut{Q}$. The following are equivalent:
\begin{itemize}
\item[(i)] $\lmlt(Q)$ normalizes $N$.
\item[(ii)] $\dis_{\mathcal{O}_N}\leq N$.
\end{itemize}
In particular, if $Q$ is a rack $\N(Q)$ is the lattice of normal subgroups of $\lmlt(Q)$.
\end{lemma}

\begin{proof}

Let $N\leq \aut{Q}$, $a\in Q$ and $n\in N$ then
$$L_a^{\pm 1} n L_a^{\mp 1}=n n^{-1} L_a^{\pm 1} n L_a^{\mp 1}=n L_{n^{-1}(a)}^{\pm 1}L_a^{\mp 1}.$$
If $\lmlt{(Q)}$ normalizes $N$, then $N$ contains the generators of $\dis_{\mathcal{O}_N}$ as a normal subgroup of $\lmlt{(Q)}$ and hence $\dis_{\mathcal{O}_N}\leq N$. On the other hand, if $\dis_{\mathcal{O}_N}\leq N$
then $L_{n^{-1}(a)}^{\pm 1} L_a^{\mp 1}\in N$ for every $n\in N$ and every $a\in Q$. Therefore $L_a^{\pm 1} n L_a^{\mp 1} \in N$ for every $n\in N$ and every $a\in Q$, i.e. $\lmlt(Q)$ normalizes $N$.

If $Q$ is a rack, then $\lmlt(Q)\leq \aut{Q}$ and so we have that every normal subgroup of $\lmlt(Q)$ is in $\N(Q)$.
\end{proof}

\section{Commutator theory}\label{Sec 2}


Let us make a brief outline of commutator theory for general algebras developed in \cite{comm, MS}. The goal of commutator theory is to define the concept of commutator for congruences, in order to define abelian and central congruence and consequently solvable and nilpotent algebraic structures.

Let $\alpha$, $\beta$, $\delta$ be congruences of an algebraic structure $A$. We say that \emph{$\alpha$ centralizes $\beta$ over $\delta$}, and write $C(\alpha,\beta;\delta)$, if for every $(n+1)$-ary term operation $t$, every pair $a\,\alpha\,b$ and every $u_1\,\beta\,v_1$, $\dots$, $u_n\,\beta\,v_n$ we have
\[  t^A(a,u_1,\dots,u_n)\,\delta \, t^A(a,v_1,\dots,v_n)\quad\Rightarrow\quad t^A(b,u_1,\dots,u_n) \, \delta \, t^A(b,v_1,\dots,v_n). \tag{TC}\]
The following observations are crucial for the definition of commutator between congruences introduced below:
\begin{itemize}
	\item[(C1)] if $C(\alpha,\beta;\delta_i)$ for every $i\in I$, then $C(\alpha,\beta;\bigwedge\delta_i)$,
	\item[(C2)] $C(\alpha,\beta;\alpha\wedge\beta)$,
	\item[(C3)] if $\theta\leq\alpha\wedge\beta\wedge\delta$, then $C(\alpha,\beta;\delta)$ in $A$ if and only if $C(\alpha/\theta,\beta/\theta;\delta/\theta)$ in $A/\theta$.
\end{itemize}

The \emph{commutator} of $\alpha$, $\beta$, denoted by $[\alpha,\beta]$, is the smallest congruence $\delta$ such that $C(\alpha,\beta;\delta)$ (the definition makes sense thanks to (C1)). From (C2) follows that $[\alpha,\beta]\leq\alpha\wedge\beta$. A congruence $\alpha$ is called 
\begin{itemize}
	\item \emph{abelian} if $C(\alpha,\alpha;0_A)$, i.e., if $[\alpha,\alpha]=0_A$,
	\item \emph{central} if $C(\alpha,1_A;0_A)$, i.e., if $[\alpha,1_A]=0_A$.
\end{itemize}
Using the centralizing relation we can define several familiar concept for arbitrary algebraic structure. The \emph{center} of $A$, denoted by $\zeta_A$, is the largest congruence $\delta$ of $A$ such that $C(\delta,1_A;0_A)$. Hence, $[\alpha,1_A]$ is the smallest congruence $\delta$ such that $\alpha/\delta\leq\zeta_{A/\delta}$. Similarly, $[\alpha,\alpha]$ is the smallest congruence $\delta$ such that $\alpha/\delta$ is an abelian congruence of $A/\delta$. 

The following lemma will be used for inductive arguments later.

\begin{lemma}\label{l:2rd_iso_thm}\cite[Lemma 4.3]{CP}
Let $A$ be an algebraic structure, and $\theta\leq\alpha\leq\beta$ its congruences. 
Then $\beta/\alpha$ is central (resp. abelian) in $A/\alpha$ if and only if $(\beta/\theta)\big/(\alpha/\theta)$ is central (resp. abelian) in $(A/\theta)\big/(\alpha/\theta)$.
\end{lemma}
%

The structure $A$ is called \emph{abelian} if $\zeta_A=1_A$, (equivalently the congruence $1_A$ is abelian). The main examples of abelian algebraic structures are modules. Affine left quasigroups are abelian left quasigroups as they are {\it reducts} of modules.

Solvability and nilpotence can be defined by the existence of chain of congruences satisfying certain centralizing relation, as familiar for groups and other classical algebraic structures. Indeed, $A$ is called \emph{nilpotent} (resp. \emph{solvable}) if there exists a chain of congruences 
\begin{displaymath}
    0_A=\alpha_0\leq\alpha_1\leq\ldots\leq\alpha_n=1_A
\end{displaymath}
such that $\alpha_{i+1}/\alpha_{i}$ is a central (resp. abelian) congruence of $A/\alpha_{i}$, for all $i\in\{0,1,\dots,n-1\}$.
The length of the smallest such series is called the \emph{length} of nilpotence (resp. solvability).

Similarly to group theory, one can define the series
\begin{displaymath}
    \gamma_{0}(A)=1_A,\qquad \gamma_{i+1}(A)=[\gamma_{i}(A),1_A],
\end{displaymath}
and
\begin{displaymath}
    \gamma^{0}(A)=1_A,\qquad \gamma^{i+1}(A)=[\gamma^{i}(A),\gamma^{i}(A)],
\end{displaymath}
and prove that an algebra $A$ is nilpotent of length $n$ (resp. solvable) if and only if $\gamma_{n}(A)=0_A$ (resp. $\gamma^{n}(A)=0_A$).  The series of the centers can be defined analogously to the series for groups as
$$\zeta_1(A)=\zeta_Q,\quad \zeta_{n+1}(A)/\zeta_n(A)=\zeta_{A/\zeta_{n}(A)}$$ 
for $n\in \mathbb{N}$ and we have that $A$ is nilpotent of length $n$ if and only if $\zeta_n(A)=1_A$.

For groups, the classical notion of abelianness and centrality of normal subgroups coincide with the corresponding notions of abelianness and centrality of the corresponding congruences. In loops, the situation is more complicated \cite{SV1}. In a wider setting, the commutator behaves well in all congruence-modular varieties \cite{comm}; for example, it is commutative (note that its definition is asymmetric with respect to $\alpha,\beta$). 


\subsection*{Commutator theory for LT left quasigroups}
The left quasigroup terms where every left branch consists of unary subterm will be called left-translation term (shortly LT term). Formally, they are terms of the form
\begin{equation}\label{t}
t(x_1,\ldots, x_n)= s_1(x_{i_1})\bullet_1( s_2(x_{i_2})
\bullet_2(\ldots(s_{m-1}(x_{i_{m-1}})\bullet_{m-1} s_m(x_{i_{m}}))\ldots)),
\end{equation} 
where $s_i$ is a unary term for $1\leq i\leq m$ and $\bullet_i\in \{*,\ldiv\}$. 
A concrete example of an LT term is given in Figure \ref{fig:ltt}.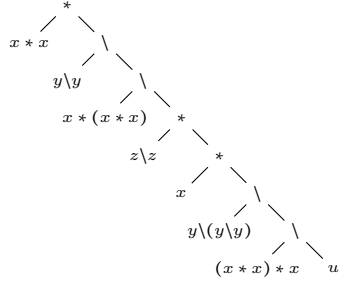
\begin{figure}[!ht]

\tiny{
\begin{tikzpicture}[scale =0.5]
%
%
%
%
%
%
%
%
%
%
\node(f) at (18,0) {$\ast$}; \node at (15,0) {};
\node(b1) at (17,-1) {$x*x$}; \node(b2) at (19,-1) {$\backslash$};
\node(b3) at (18,-2) {$y\ldiv y$}; \node(b4) at (20,-2) {$\backslash$};
\node(b5) at (19,-3) {$x*(x*x)$}; \node(b6) at (21,-3) {$\ast$};
\node(b7) at (20,-4) {$z\ldiv z$}; \node(b8) at (22,-4) {$\ast$};
\node(b9) at (21,-5) {$x$}; \node(b10) at (23,-5) {$\backslash$};
\node(b11) at (22,-6) {$y\ldiv (y\ldiv y)$}; \node(b12) at (24,-6) {$\backslash$};
\node(b13) at (23,-7) {$(x*x)*x$}; \node(b14) at (25,-7) {$u$};

\draw(f) -- (b1);
\draw(f) -- (b2);
\draw(b2) -- (b3);
\draw(b2) -- (b4);
\draw(b4) -- (b5);
\draw(b4) -- (b6);
\draw(b6) -- (b7);
\draw(b6) -- (b8);
\draw(b8) -- (b9);
\draw(b8) -- (b10);
\draw(b10) -- (b11);
\draw(b10) -- (b12);
\draw(b12) -- (b13);
\draw(b12) -- (b14);

\end{tikzpicture} }
\caption{An example of LT term.}\label{fig:ltt}
\end{figure}

A left-quasigroup $Q$ has the {\it LT property} if every term is equivalent to a LT-term ($Q$ is also said to be a {\it left translation term} (LT left quasigroup). Therefore $Q$ is a LT left-quasigroup if and only if for every term $t$ an identity as in \eqref{t} holds in $Q$. Clearly such identities are also satisfied by any subalgebra, factor or power of $Q$ and so they also have the LT property.

\begin{remark*}
The class of LT left-quasigroups is stable under taking subalgebras, homomorphic images and powers.
\end{remark*}
Examples of LT left quasigroups are racks (see \cite{CP}). 
Indeed, using \eqref{LD} we can transform every term into an LT term (see Figure \ref{tree for rack}).

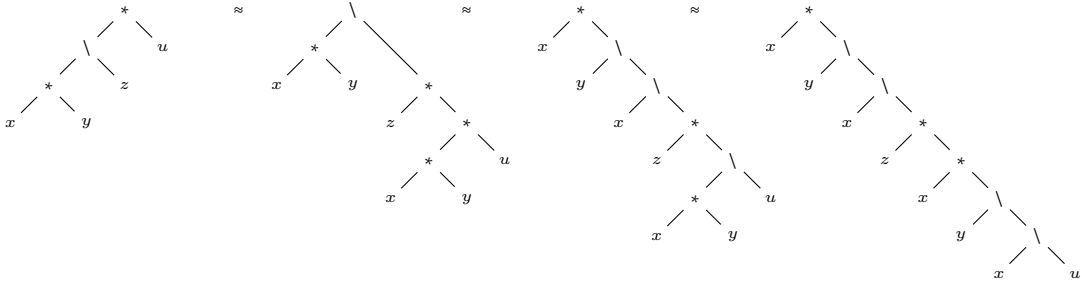
\begin{figure}[!ht]

\tiny{
\begin{tikzpicture}[scale =0.5]
\node(a0) at (0,0) {$\ast$};  \node(m) at (3,0) {$\approx$}; 
\node(a1) at (-1,-1) {$\backslash$}; \node(a2) at (1,-1) {$u$}; 
\node(b1) at (-2,-2) {$\ast$}; \node(b2) at (0,-2) {$z$}; 
\node(c1) at (-3,-3) {$x$}; \node(c2) at (-1,-3) {$y$}; 

\node(aa0) at (6,0) {$\backslash$}; 
\node(aa1) at (5,-1) {$\ast$};
\node(bb1) at (4,-2) {$x$}; \node(bb2) at (6,-2) {$y$};
 \node(cc1) at (8,-2) {$\ast$};
\node(d1) at (7,-3) {$z$};  \node(d2) at (9,-3) {$\ast$};
\node(e1) at (8,-4) {$\ast$};  \node(e2) at (10,-4) {$u$};
\node(f1) at (7,-5) {$x$};  \node(f2) at (9,-5) {$y$};

\draw(aa0) -- (aa1);
\draw(aa1) -- (bb1);
\draw(aa0) -- (cc1);
\draw(cc1) -- (d1);

\draw(e1) -- (f1);

\draw(e1) -- (f2);

\draw(cc1) -- (d2);
\draw(d2) -- (e1);
\draw(d2) -- (e2);

\draw(aa1) -- (bb2);
\draw(a0) -- (a1);
\draw(a0) -- (a2);
\draw(a1) -- (b1);
\draw(a1) -- (b2);
\draw(b1) -- (c1);
\draw(b1) -- (c2);

\node(a0) at (12,0) {$\ast$}; \node at (9,0) {$\approx$};
\node(b0) at (11,-1) {$x$}; \node(b1) at (13,-1) {$\backslash$};
\node(c0) at (12,-2) {$y$}; \node(c1) at (14,-2) {$\backslash$};
\node(dd0) at (13,-3) {$x$}; \node(dd1) at (15,-3) {$\ast$};
\node(ee0) at (14,-4) {$z$}; \node(ee1) at (16,-4) {$\backslash$};
\node(ff0) at (15,-5) {$\ast$}; \node(ff1) at (17,-5) {$u$};
\node(gg0) at (14,-6) {$x$}; \node(gg1) at (16,-6) {$y$};

\draw(a0) -- (b0);
\draw(a0) -- (b1);
\draw(b1) -- (c0);
\draw(b1) -- (c1);
\draw(c1) -- (dd0);
\draw(c1) -- (dd1);
\draw(dd1) -- (ee0);
\draw(dd1) -- (ee1);
\draw(ee1) -- (ff0);
\draw(ee1) -- (ff1);
\draw(ff0) -- (gg0);
\draw(ff0) -- (gg1);

\node(f) at (18,0) {$\ast$}; \node at (15,0) {$\approx$};
\node(b1) at (17,-1) {$x$}; \node(b2) at (19,-1) {$\backslash$};
\node(b3) at (18,-2) {$y$}; \node(b4) at (20,-2) {$\backslash$};
\node(b5) at (19,-3) {$x$}; \node(b6) at (21,-3) {$\ast$};
\node(b7) at (20,-4) {$z$}; \node(b8) at (22,-4) {$\ast$};
\node(b9) at (21,-5) {$x$}; \node(b10) at (23,-5) {$\backslash$};
\node(b11) at (22,-6) {$y$}; \node(b12) at (24,-6) {$\backslash$};
\node(b13) at (23,-7) {$x$}; \node(b14) at (25,-7) {$u$};

\draw(f) -- (b1);
\draw(f) -- (b2);
\draw(b2) -- (b3);
\draw(b2) -- (b4);
\draw(b4) -- (b5);
\draw(b4) -- (b6);
\draw(b6) -- (b7);
\draw(b6) -- (b8);
\draw(b8) -- (b9);
\draw(b8) -- (b10);
\draw(b10) -- (b11);
\draw(b10) -- (b12);
\draw(b12) -- (b13);
\draw(b12) -- (b14);

\end{tikzpicture} }
\caption{Transforming the term $((x*y)\ld z)*u$ into a left translation form using \eqref{LD}.}\label{tree for rack}
\end{figure}

In \cite{CP} a characterization of the centralizing relation has been proved for the class of LT left-quasigroups, providing also a class in which commutator theory is particularly easy to understand and which is not within the framework of congruence modular varieties. To this end we need the notion of $\alpha$-semiregularity of a group of permutation $N$ acting on a set $Q$: if $\alpha$ is an equivalence relation we say that $N$ is {\it $\alpha$-semiregular} if whenever $h(a)=a$ for some $n\in N$ and $a\in Q$ then $h(b)=b$ for every $b\, \alpha\, a$. We define the relation $\sigma_N$ as
$$a\, \sigma_N\, b\, \,\text{ if and only if }\, N_a=N_b.$$
The group $N$ is $\alpha$-semiregular if and only if $\alpha\leq \sigma_{N}$.

\begin{proposition}\label{p:commutators}\cite[Proposition 5.2]{CP}
Let $Q$ be a LT left quasigroup and let $\alpha,\beta\in Con(Q)$. Then $[\alpha,\beta]$ is the smallest congruence $\delta$ such that $[\dis_{\alpha/\delta},\dis_{\beta/\delta}]=1$ and $\dis_{\beta/\delta}$ acts $\alpha/\delta$-semiregularly on $Q/\delta$.
\end{proposition}
\noindent The forward implication of Proposition \ref{p:commutators} holds for arbitrary left quasigroups (see \cite[Lemma 5.1]{CP}).

\begin{corollary}\label{ab an central cong iff}\cite[Theorem 1.1]{CP}
Let $Q$ be a LT left quasigroup and $\alpha\in Con(Q)$.
\begin{itemize}
\item[(i)] $\alpha$ is abelian if and only if $\dis_\alpha$ is abelian and $\alpha$-semiregular.
\item[(ii)] $\alpha$ is central if and only if $\dis_\alpha$ is central in $\dis(Q)$ and $\dis(Q)$ is $\alpha$-semiregular.
\end{itemize} 
\end{corollary}
%
%



%
%

\begin{corollary}\label{central cong}
Let $Q$ be a LT left quasigroup and $\alpha\in Con(Q)$. The following are equivalent: 
\begin{itemize}
\item[(i)] $\alpha$ is central. 
\item[(ii)] $\alpha\leq \c{Z(\dis(Q))}\cap \sigma_{\dis(Q)}$. 
\end{itemize}
In particular, if $\c{Z(\dis(Q))}$ is a congruence then $\zeta_Q=\c{Z(\dis(Q))}\cap \sigma_{\dis(Q)}$.
\end{corollary}

\begin{proof}
(i) $\Rightarrow$ (ii) If $\alpha$ is central, then $\dis_\alpha\leq Z(\dis(Q))$ and so $\alpha\leq \c{\dis_\alpha}\leq \c{Z(\dis(Q))}$. Moreover, if $h(a)=a$ then $h(b)=b$ for every $a\, \alpha\, b$ and every $h\in \lmlt(Q)$, i.e. $\alpha\leq \sigma_{\dis(Q)}$. 

(ii) $\Rightarrow$ (i) If $\alpha\leq \c{Z(\dis(Q))}\cap \sigma_{\dis(Q)}$, then $\dis_\alpha\leq Z(\dis(Q))$ and if $h(a)=a$ then $h(b)=b$ for $a\,\alpha\,b$, i.e. $\dis(Q)$ is semiregular.

Assume that the relation $\c{Z(\dis(Q))}$ is a congruence. If we prove that $\beta=\c{Z(\dis(Q))}\cap \sigma_{\dis(Q)}$ is a congruence then it is equal to $\zeta_Q$. We just need to prove that $\dis(Q)_{L_a^{\pm 1}(c)}=\dis(Q)_{L_b^{\pm 1}(c)}$ and $\dis(Q)_{L_c^{\pm 1}(a)}=\dis(Q)_{L_c^{\pm 1}(b)}$ whenever $a\, \beta\, b$. The first equality follows since $L_b^{\mp 1} L_a^{\pm 1}\in Z(\dis(Q))$ and so $$\dis(Q)_{L_a^{\pm 1}(c)}=L_a^{\pm 1} \dis(Q)_c L_a^{\mp 1}=L_b^{\pm 1} L_b^{\mp 1} L_a^{\pm 1}\dis(Q)_c L_a^{\mp 1}=L_b^{\pm 1}\dis(Q)_c L_b^{\mp 1}=\dis(Q)_{L_b^{\pm 1}(c)}.$$
The second one follows since
$$\dis(Q)_{L_c^{\pm 1}(a)}=L_c^{\pm 1} \dis(Q)_a L_c^{\mp 1}=L_c^{\pm 1} \dis(Q)_b L_c^{\mp 1}=\dis(Q)_{L_c^{\pm 1}(b)}. $$
\end{proof}

Note that, $\c{Z(\dis(Q))}\cap \sigma_{\dis(Q)}$ is not a congruence of $Q$ in general. Nevertheless, this is the case for racks  \cite[Proposition 5.9]{CP}. Let
\medskip
\begin{center}
$Q=$\,\begin{tabular}{|c c c c  |}
\hline
 2 & 1 & 3 & 4 \\
  1 &  2 & 4  & 3 \\
   2 & 1 & 3 & 4\\
   3 & 4 & 2 & 1\\
\hline
\end{tabular}\,.
\end{center}
\comment{is this LT?}

\medskip

The relation $\alpha=\c{Z(\dis(Q))}\cap \sigma_{\dis(Q)}$ 
is not a congruence (indeed $1\,\alpha\,2$ but $1*3=3$ and $2*3=4$ are not $\alpha$-related).


Abelianness and centrality of congruence coming from orbit decompositions of subgroup of the left multiplication are reflected by the properties of the subgroup.

\begin{lemma}\label{abelian subgroup gives abelian cong for LTT}
Let $Q$ be a LT left-quasigroup and let $N\in \N(Q)$ abelian (resp. centralizes $\dis(Q)$). Then $\O_N$ is an abelian (resp. central) congruence of $Q$. 
\end{lemma}

\begin{proof}
According to Corollary \ref{ab an central cong iff}, we need to check that $\dis_{\O_N}$ is abelian (resp. central in $\dis(Q)$) and that $\dis_{\O_N}$ (resp. $\dis(Q)$) acts $\O_N$-semiregularly on $Q$.
Since $\dis_{\O_N}\leq N$ then it is abelian (resp. central in $\dis(Q)$). Let $h\in \dis_{\O_N}$ (resp. $h\in \dis(Q)$) and let $h(a)=a$. If $b=n(a)$ for some $n\in N$, then $hn(a)=nh(a)=n(a)$, therefore $\dis_{\O_N}$ (resp. $\dis(Q)$) is $\O_N$-semiregular.
%
\end{proof}

Let $G$ be a group. We say that 
$$1=H_0\leq H_1\leq H_2\leq \ldots \leq H_n=G$$
is a {\it central} (resp. {\it abelian}) series if $H_i\unlhd G$ and $H_{i+1}/H_i\leq Z(G/H_i)$ (resp. $H_{i+1}/H_i$ is abelian) for every $i\in \{ 0,\ldots, n-1\}$. 

Let $K\unlhd G$ and $\pi_K:G\longrightarrow G/K$ the canonical projection. Then
$$[H_{i+1} K, G]\leq H_{i} K \quad \text{(resp. }[H_{i+1} K, H_{i+1} K]\leq H_{i} K )$$
for every $i\in \{0,\ldots n-1\}$ and therefore $[\pi_K(H_{i+1}),G/K]\leq \pi_K(H_i)$ (resp. $[\pi_K(H_{i+1}),\pi_K(H_{i+1})]\leq \pi_K(H_i)$). So the series
  $$1\leq \pi_K(H_2)\leq \pi_K(H_3)\leq \ldots\leq \pi_K(H_{n-1})\leq G/K$$
 is a central (resp. abelian) series of $G/K$.

\begin{proposition}\label{dis solv then Q solvable for LT}
Let $Q$ be a LT left quasigroup. If
$$1=H_0\leq H_1\leq H_2\leq \ldots \leq H_n=\dis(Q)$$
is an abelian (central) series and $H_i\in \N(Q)$ for every $i\in \{1,\ldots,n-1\}$ then $Q$ is solvable (nilpotent) of length less or equal to $n+1$.
\end{proposition}

\begin{proof}
We proceed by induction on $n$. 
For $n=1$, the group $\dis(Q)$ is abelian, and so the orbit decomposition of $Q$ is a central congruence and the factor is a permutation left-quasigroup and so abelian.  Then $Q$ is nilpotent (and thus solvable, too) of length $\leq 2$.
In the induction step, assume that the statement holds for all LT left quasigroups with an central (abelian) series of length at most $n-1$.


 Since $H_1$ is central (resp. abelian) in $\dis(Q)$, the congruence $\alpha=\O_{H_1}$ is also central (resp. abelian), by Lemma \ref{abelian subgroup gives abelian cong for LTT}. 
The series
  $$1\leq \pi_\alpha(H_2)\leq \pi_\alpha(H_3)\leq \ldots\leq \pi_\alpha(H_{n-1})\leq \dis(Q/\alpha)$$
 is a central (resp. abelian) series of length $n-1$ and according to Lemma \ref{N(Q)} its elements are in $\N(Q/\alpha)$. By the induction assumption, $Q/\alpha$ is nilpotent (resp. solvable) of length $m\leq n$.
Let
\[ 0_{Q/\alpha}\leq\alpha_1/\alpha\leq\ldots\leq \alpha_m/\alpha=1_{Q/\alpha} \]
be the witness. Then 
\[ 0_Q\leq\alpha\leq\alpha_1\leq\ldots\leq \alpha_m=1_Q \]
is the witness that $Q$ is nilpotent (resp. solvable) of length less or equal to $n+1$, using Lemma \ref{l:2rd_iso_thm}.
\end{proof}

\section{Semimedial left quasigroups}\label{Sec 3}
%

%
%
%


A left quasigroup $Q$ is called {\it (left) semimedial} if one of the following equivalent identity
\begin{align}
(x*x)*(y*z)&\approx (x*y)*(x*z), \label{semi M} \tag{SM1} \\
(x\ldiv y)*(x\ldiv z)&\approx (x*x)\ldiv (y*z), \label{eq semimedial 2} \tag{SM2}
\end{align}
holds. Examples of semimedial left quasigroups are racks and medial left quasigroups (e.g. abelian groups). Idempotent semimedial left quasigroups are quandles.


Note that \eqref{semi M} and \eqref{eq semimedial 2} are equivalent to the following identities for the left multiplications mapping
\begin{equation}\label{semi for mappings}
L_{L_x^{\pm 1}(y)}=L_{x*x}^{\pm 1} L_y L_x^{\mp 1}.
\end{equation}

If a left quasigroup $Q$ satisifes an identity
$$f(x)*(y*z)\approx (x*y)*(x*z),$$
for some mapping $f:Q\longrightarrow Q$, then $L_{f(x)}=L_{x*x}$ and so $Q$ is semimedial.

%
%

Let $Q$ be a semimedial left quasigroup and let $h=L_{a_1}^{k_1}\ldots L_{a_n}^{k_n}\in \lmlt(Q)$. We define 
$$\! h=L_{\s(a_1)}^{k_1}\ldots L_{\s(a_n)}^{k_n}$$
and we denote by $\! S=\setof{\! h}{h\in S}$ for every $S\subseteq \lmlt(Q)$. The assignment $h\mapsto  \!{h}$ has the following properties:
\begin{eqnarray*}
\!{(h^{-1})}=(\! h)^{-1}, \quad \!{(hg)}=\!h \!g
\end{eqnarray*}
for every $h,g\in \lmlt(Q)$. In particular, if $N$ is a subgroup of $\lmlt(Q)$ then so it is $\! N$. Let us remark that the assignment $h\mapsto \!{h}$ is well defined on the level of words in the free group generated by $\setof{L_a}{a\in Q}$, not as a function of $\lmlt(Q)$. Note that for racks $h^\s=h$.

\begin{lemma}\label{on dis_alpha}
Let $Q$ be a semimedial left-quasigroup and $\alpha\in Con(Q)$. Then

\begin{itemize}

\item[(i)] If $h\in \lmlt(Q)$ then \begin{equation}\label{mediality2}
L_{h(a)}=\! h L_a h^{-1}
\end{equation}
for every $a\in Q$.
\item[(ii)] $\N(Q)=\setof{N\trianglelefteq \lmlt(Q)}{N^\s\leq N}$.

\item[(iii)] $\dis_\alpha=\langle L_a^{-1} L_b,\, a\, \alpha\, b\rangle $.
\item[(iv)] $[\lmlt(Q),\lmlt^\alpha]\leq \dis_\alpha$.  
%
\end{itemize}

\end{lemma}
%

\begin{proof}
(i) Let $h=L_{a_1}^{k_1}\ldots L_{a_n}^{k_n}$. We proceed by induction on $n$. For $n=1$, we can apply directly formula \eqref{semi for mappings}.
Let $h=L_b^{\pm 1} g$. Hence by induction we have
$$L_{h(a)}=L_{b*b}^{\pm 1} L_{g(a)} L_b^{\mp 1}=L_{b*b}^{\pm 1} \! g L_a g^{-1} L_b^{\mp 1}=\! h L_a h^{-1}$$ 
for every $a\in Q$. 
%
%
%
%

(ii) Let $N\trianglelefteq \lmlt(Q)$. Using (i), we have
$$L_{n(a)}L_a^{-1}=\!{n} \underbrace{L_a n L_a^{-1}}_{\in N}$$
for every $n\in N$ and every $a\in Q$. Hence $L_{n(a)} L_a^{-1}\in N$ if and only if $\!{n}\in N$.

(iii) Let $D=\langle L_a^{-1} L_b , \, a\, \alpha\, b\rangle$. We have that \eqref{semi for mappings} implies that 
\begin{equation}\label{inclusion dis}
L_a L_b^{-1}=L_{b*b}^{-1} L_{b*a} \in D\leq \dis_\alpha.
\end{equation}
So, it is enough to prove that $D$ is normal in $\lmlt(Q)$. Let $a\, \alpha\, b $ and $x\in Q$. Then
\begin{eqnarray*}
L_x^{\pm 1} L_a^{-1} L_b L_x^{\mp 1}&=& L_x^{\pm 1} L_{a}^{-1}L_{x*x}^{\mp 1} L_{x*x}^{\pm 1} L_{b} L_x^{\mp 1}\\
&=& L_{L_x^{\pm 1}(a)}^{-1} L_{L_x^{\pm 1}(b)}\in D.
\end{eqnarray*}
Thus, $D$ is normal in $\lmlt(Q)$.

(iv) Let $h\in \lmlt^{\alpha}$ and $a\in Q$. Then $h(a)\, \alpha\, a$ for every $a\in Q$ and so
\begin{eqnarray*}
[L_a^{\mp 1},h]&=&L_a^{\pm 1} h^{-1} L_a^{\mp 1} h= (\! {h})^{-1} \! h L_a^{\pm 1} h^{-1} L_a^{\mp 1} h\\
&=&(\! {h})^{-1}  L_{h(a)}^{\pm 1} L_a^{\mp 1} h\\
&=& L_a^{\pm 1} L_{h^{-1}(a)}^{\mp 1}\in \dis_\alpha.
\end{eqnarray*}

Since $\lmlt^\alpha$ and $\dis_\alpha$ are normal subgroup of $\lmlt(Q)$, then we have $[\lmlt(Q),\lmlt^\alpha]\leq \dis_\alpha$. 
%
\end{proof}

%
%
%
%
%

\begin{proposition}\label{terms for medial} 
Semimedial left quasigroups have the LT property.
\end{proposition}

\begin{proof}
Let $t=q\bullet r$ be a term where $\bullet\in \{*,\ldiv\}$ and $q$ and $r$ are subterms of $t$. By induction on the number of occurrences $q$ and $r$ are LT terms and so using \eqref{mediality2} we have
\begin{eqnarray*}
t(x_1,\ldots,x_n)&=&L_{q(x_{1},\ldots,x_{n})}^{\pm 1}(r(x_{1},\ldots, x_{n}))\\
&=& \left(L_{L_{t_1(x_{i_1})}^{k_1}\ldots L_{t_m(x_{i_m})}^{k_m}(x_{i_{m+1}} )}\right)^{\pm 1} L_{u_1(x_{j_1})}^{v_1}\ldots L_{u_s(x_{j_s})}^{v_s} (x_{j_{s+1}} )\\
&=& \left(L_{\s (t_1(x_{i_1}))}^{k_1}\ldots L_{\s( t_m(x_{i_m}))}^{k_m} L_{x_{i_{m+1}} }(L_{t_1(x_{i_1})}^{k_1}\ldots L_{t_m(x_{i_m})}^{k_m})^{-1}\right)^{\pm 1}L_{u_1(x_{j_1})}^{v_1}\ldots L_{u_s(x_{j_s})}^{v_s} (x_{j_{s+1}} ),
\end{eqnarray*}
where $x_{i_l}, x_{j_l}\in \{x_1,\ldots,x_n\}$ and $t_i,u_j$ are unary terms for every $i,l\in \{1,\ldots,m+1\}$, $j,l\in \{1,\ldots,s+1\}$. Therefore $t$ is an LT term. 
%
\end{proof}

\begin{remark}\label{like spelling}
We can establish the LT property for all left quasigroups satisfying a condition similar to \eqref{mediality2}. Indeed, if $Q$ is a left quasigroup and for every $h=L_{x_1}^{k_1}\ldots L_{x_n}^{k_n}\in \lmlt(Q)$ there exist unary terms $t_1,\ldots, t_n$, $\{y_1,\ldots, y_m\}\subseteq \{x_1,\ldots ,x_m\}$ and $\{s_1,\ldots s_m\}\subseteq \mathbb{Z}$ such that 
$$L_{h(a)}=L_{t_1(y_1)}^{s_1}\ldots L_{t_n(y_m)}^{s_m}$$ 
then the very same argument of Proposition \ref{terms for medial} applies and so $Q$ has the LT property.
\end{remark}

Proposition \ref{p:commutators} and Corollary \ref{ab an central cong iff} apply to semimedial left quasigroup as they have the LT property. We can characterize solvable and nilpotent semimedial left quasigroups by properties of the corresponding displacement groups. In the following we will denote by $\Gamma_{(n)}$ (resp. $\Gamma^{(n)}$) the $n$-th element of the lower central series (resp. the derived series) of a group. The proof follows the very same argument as in Section 6 of \cite{CP}. 
\begin{lemma}\label{solvable then dis solv} 
Let $Q$ be a semimedial left quasigroup. 

\begin{itemize}
\item[(i)]  If $\dis(Q)$ is nilpotent (resp. solvable) of length $n$, then $Q$ is nilpotent (resp. solvable) of length less or equal to $ n+1$.

\item[(ii)] If $Q$ is nilpotent (resp. solvable) of length $n$, then $\dis(Q)$ is a nilpotent (resp. solvable) group of length less or equal to $ 2n-1$.
%
\end{itemize}
\end{lemma}

\begin{proof}

(i) The subgroups of the derived and of the lower central series of $\dis(Q)$ are in $ \N(Q)$. Indeed $[h^\s,g^\s]=[h,g]^\s$ for every $h,g\in \dis(Q)$. Hence, we can apply Proposition \ref{dis solv then Q solvable for LT}. 

(ii) We proceed by induction on the length $n$. For $n=1$, then $Q$ is abelian, hence $\dis(Q)$ is an abelian group and the statement holds.
In the induction step, assume that the statement holds for all semimedial left quasigroups that are nilpotent (resp. solvable) of length $\leq n-1$.
Consider a chain of congruences \[ 0_{Q}=\alpha_0\leq\alpha_1\leq\ldots\leq \alpha_n= 1_{Q}\]
such that $\alpha_{i+1}/\alpha_i$ is central (resp. abelian) in $Q/\alpha_i$, for every $i$. In particular, $\alpha_1$ is central (resp. abelian) in $Q$ and the factor $Q/\alpha_1$ is nilpotent (resp. solvable) of length $n-1$, as witnessed by the series 
\[ 0_{Q/\alpha_1}=\alpha_1/\alpha_1\leq\alpha_2/\alpha_1\leq\ldots\leq \alpha_n/\alpha_1=1_{Q/\alpha_1} \]
(see Lemma \ref{l:2rd_iso_thm}). By the induction assumption, $\dis(Q/\alpha_1)$ is nilpotent (resp. solvable) of length $\leq2n-3$.
Now, consider the series $\Gamma_{(i)}$ (resp. $\Gamma^{(i)}$) in $\dis(Q)$ and project it into $\dis(Q/\alpha_1)$. Since $\pi_{\alpha_1}(\Gamma_{(2n-3)})=1$, we obtain that
$\Gamma_{(2n-3)}\leq \dis^{\alpha_1}$ (resp. analogically for $\Gamma^{(2n-3)}$). Now, in case of nilpotence, we have
\begin{align*}
\Gamma_{(2n-1)} &= [[\Gamma_{(2n-3)},\dis(Q)],\dis(Q)] \\
&\leq [[ \dis^{\alpha_1},\dis(Q)],\dis(Q)] \leq [\dis_{\alpha_1},\dis(Q)]=1,
\end{align*}
using Lemma \ref{on dis_alpha}(iv) in the penultimate step, and centrality of $\dis_{\alpha_1}$ in the ultimate step. In case of solvability, we have
\begin{align*}
\Gamma^{(2n-1)} &= [[\Gamma^{(2n-3)},\Gamma^{(2n-3)}],[\Gamma^{(2n-3)},\Gamma^{(2n-3)}]] \\
&\leq [[ \dis^{\alpha_1},\dis^{\alpha_1}],[ \dis^{\alpha_1},\dis^{\alpha_1}]] \leq [\dis_{\alpha_1},\dis_{\alpha_1}]=1, 
\end{align*}
using Lemma \ref{on dis_alpha}(iv), and abelianness of $\dis_{\alpha_1}$.
%
\end{proof}

The relation $\c{N}$ determined by any subgroup $N$ of $\lmlt(Q)$ is a congruence and accordingly we have a second Galois connection between $Con(Q)$ and $\N(Q)$ for semimedial left quasigroups, extending the same result known for racks \cite[Proposition 3.6]{CP}.

\begin{theorem}\label{Galois for semimedial}
Let $Q$ be a semimedial left quasigroup. The assignment $\alpha\mapsto\dis_\alpha$ and $N\mapsto\c{N}$ is a monotone Galois connection between $Con(Q)$ and $\N(Q)	$.
\end{theorem}

\begin{proof}
First we need to show that $\c{N}$ is a congruence for every $N\in \N(Q)$. Assume that $L_a L_b^{-1}, \, L_c L_d^{-1}\in N$. Then since $N$ is normal in $\lmlt(Q)$ then also $L_a^{-1}L_b\in N$. So we have
\begin{eqnarray*}
L_{a*c} L_{b*d}^{-1}&=&L_{a*a} L_c L_a^{-1} L_{b} L_d^{-1} L_{b*b}^{-1}=\underbrace{ L_{a*a} L_{b*b}^{-1} }_{\!{(L_a L_b^{-1})}\in N} L_{b*b} L_c \underbrace{L_a^{-1} L_b}_{\in N} L_c^{-1} \underbrace{L_c L_d^{-1}}_{\in N} L_{b*b}^{-1} \in N,\\
L_{a\ldiv c} L_{b\ldiv d}^{-1}&=&L_{a*a}^{-1} L_c L_a L_{b}^{-1} L_d^{-1} L_{b*b}=\underbrace{ L_{a*a}^{-1} L_{b*b}}_{\!{(L_a^{-1} L_b)}\in N} L_{b*b} ^{-1} L_c \underbrace{L_a L_b^{-1}}_{\in N} L_c^{-1} \underbrace{L_c L_d^{-1}}_{\in N} L_{b*b} \in N,\\
\end{eqnarray*}
where we used that $N$ is normal in $\lmlt(Q)$ and that $\! N  \leq N$. Hence $\c{N}$ is a congruence. 

If $\alpha\leq \c{N}$ then $L_a L_b^{-1}\in N$ whenever $a\,\alpha\,b$ and then $\dis_\alpha\leq N$. If $\dis_\alpha\leq N$ then $L_a L_b^{-1}\in N$ whenever $a\,\alpha\,b$. Then $\alpha\leq \c{N}$. Both the assignments $\alpha\mapsto \dis_\alpha$ and $N\mapsto \c{N}$ are clearly monotone, hence this pair of mappings provides a monotone Galois connection.
%
\end{proof}

The Cayley kernel of semimedial left quasigroups is a congruence.

\begin{proposition}\label{kernel of lambda}  
Let $Q$ be a semimedial left-quasigroups. Then $Q$ has the Cayley property and $\lmlt^{\lambda_Q}\leq Z(\lmlt(Q))$.
\end{proposition}

\begin{proof}
Clearly if $L_a=L_b$ then $L_{L_a^{\pm 1}(c)}=L_{L_b^{\pm 1}(c)}$ holds for every $c\in Q$. According to \eqref{mediality2} then $L_{c*a}=L_{c*c}L_aL_c^{-1}=L_{c*c}L_b L_c^{-1}=L_{c*b}$ and similarly $L_{c\ldiv a}=L_{c\ldiv b}$ holds for every $c\in Q$. Therefore $\lambda_Q$ is a congruence of $Q$.
%
%
%
%
%
According to Lemma \ref{on dis_alpha}(iii) we have that $[\lmlt(Q),\lmlt^{\lambda_Q}]\leq \dis_{\lambda_Q}=1$. Hence, $\lmlt^{\lambda_Q}\leq Z(\lmlt(Q))$.
%
\end{proof}

Reductive semimedial left quasigroups are strongly solvable in the sense of \cite{comm}, according to \cite[Theorem 5.3]{covering_paper}. As noticed in the case of racks,  they are also nilpotent.

\begin{corollary}\label{multip}
Let $Q$ be a semimedial $n$-reductive left quasigroup. Then $\lmlt(Q)$ (resp. $\dis(Q)$) is nilpotent of length less or equal to $n$ (resp. $n-1$). In particular, $Q$ is nilpotent of length less or equal to $n$.
\end{corollary}

\begin{proof}
If $n=1$ then $\lmlt(Q)$ is cyclic and then abelian. By induction, $\lmlt(Q/\lambda_Q)$ is nilpotent of length $n-1$ and since $\lmlt^{\lambda_Q}$ is central then $\lmlt(Q)$ is nilpotent of length $n$. The second statement can be proved by the same argument, using as a base step for the induction that the displacement group of permutation left quasigroup is trivial.

Finally, we can apply Lemma \ref{solvable then dis solv} to obtain that $Q$ is nilpotent of length at most $n$.
\end{proof}

The converse of Corollary \ref{multip} does not hold: every abelian group is semimedial, has abelian left multiplication group but it is faithful.

Commutator theory can be simplified for faithful semimedial left quasigroups, similarly to what happens for faithful quandles. All the following results basically extend the results in \cite[Section 5.2]{CP} to semimedial left quasigroups. Most of the argument can be applied again also in the setting of semimedial left quasigroups, so we include the proof just in case it is new and we refer to the analogous result for quandles otherwise.

\begin{lemma}\label{for faith}\cite[Lemma 5.3]{CP}
Let $Q$ be a faithful semimedial left quasigroup and $\alpha\in Con(Q)$. If $[H,\dis_\alpha]=1$ then $H$ is $\alpha$-semiregular.
\end{lemma}
\begin{proof}
Let $h\in H$. Then
$$h L_a^{-1} L_b h^{-1}=h L_a^{-1}(h^{\s})^{-1} h^{\s}  L_b h^{-1}=L_{h(a)}^{-1} L_{h(b)}=L_a^{-1} L_b$$
whenever $a\, \alpha\, b$. If $h(a)=a$ then $L_{h(b)}=L_b$ and since $Q$ is faithful then $h(b)=b$. Therefore, $H$ is $\alpha$-semiregular.
\end{proof}

\begin{corollary}\cite[Corollary 5.4]{CP} 
Let $Q$ be a faithful semimedial left quasigroup and $\alpha\in Con(Q)$. Then $\alpha$ is central (resp. abelian) if and only if $\dis_\alpha$ is central in $\dis(Q)$ (resp. abelian).
\end{corollary}
%

\begin{proposition}\label{p:comm_comm} \cite[Propositions 3.8, 5.5]{CP}
Let $Q$ be a semimedial left quasigroup such that every factor of $Q$ is faithful. Then the $\dis$ operator is injective and the $\c{}$ operator is surjective and 
 \[ [\alpha,\beta]=[\beta,\alpha]=\c{[\dis_{\alpha},\dis_\beta]} \]
for every $\alpha,\beta\in Con(Q)$.
\end{proposition}


%
%
%
\subsection*{Multipotent semimedial left quasigroup}

%

For semimedial left-quasigroup the squaring mapping $\s$ is an endomorphism. Indeed, if $Q$ is a semimedial left quasigroup, using \eqref{semi M} we have
\begin{equation}\label{squaring for medial}
\s(a\ast b)=(a\ast b)\ast (a\ast b)=(a\ast a)\ast (b\ast b)=\s(a)\ast \s(b),
\end{equation}
for every $a,b\in Q$. Since endomorphisms with respect to $*$ are also endomorphism with respect to $\ldiv$ then $\s$ is an endormorphism of $(Q,*,\ldiv)$. In particular, $\ker{\s}$ is a congruence of $Q$ and $E(Q)=\setof{a\in Q}{a*a=a}$ is a subalgebras since it is given by the fixed elements of the endomorphism $\s$. Moreover, since $\ker{\s}\cap \lambda_Q=0_Q$, finite {\it subdirectly irreducible} semimedial left quasigroups are either $2$-divisible or faithful.

%
%
%
%
%
%
%
%
%
%

\begin{lemma}\label{fixed points for unipotent}
Let $Q$ be a semimedial left quasigroup. If $\s(a)=\s(b)$ and $a\neq b$ then $L_a^{-1} L_b$ is fixed-point-free.
\end{lemma}

\begin{proof}
If $a*a=b*b$ then by \eqref{semi for mappings} it follows that
$$L_{a*x}^{-1} L_{b*x}=L_a L_x^{-1} L_{a*a}^{-1} L_{b*b} L_x L_b^{-1}=L_a L_b^{-1}$$ 
for every $x\in Q$. If $L_a^{-1} L_b(x)=x$, then $a*x=b*x$ and so $L_a=L_b$. Hence $a=b$ by Lemma \ref{s and lambda}.
\end{proof}

\begin{proposition}\label{multipotent have right mult 1-1}
Multipotent semimedial left quasigroups have injective right multiplications. In particular, they are faithful. 
\end{proposition}

\begin{proof}
Let $n=1$, i.e. $Q$ is a unipotent semimedial left quasigroup. According to Lemma \ref{fixed points for unipotent}, $L_a^{-1} L_b$ has no fixed points for every $a\neq b\in Q$. Therefore, since
$$a*x=b*x\, \Leftrightarrow\, L_a^{-1} L_b(x)=x,$$
then $R_x$ is injective for every $x\in Q$.
%

Let $n>1$. By induction, the right multiplications of $\s(Q)$ are injective. Hence, if $a*x=b*x$ then $\s(a)*\s(x)=\s(b)*\s(x)$ and then $\s(a)=\s(b)$. Therefore, using Lemma \ref{fixed points for unipotent}, $L_b^{-1} L_a$ has no fixed points, whenever $a\neq b$. Hence $a=b$ i.e. $R_x$ is injective.

According to Lemma \ref{injective right}, multipotent semimedial left quasigroups are faithful.
\end{proof}
%
%
%
%

\begin{corollary}
Finite multipotent semimedial left quasigroup are quasigroup.
\end{corollary}

\begin{remark}
Let $Q$ be a unipotent semimedial left quasigroup and let $E(Q)=\{e\}$. Then 
$$(L_{a_1}^{k_1}\ldots L_{a_n}^{k_n})^\s=L_e^{ k_1+k_2+\ldots+k_n}.
$$ 
If $h\in \dis(Q)$ then $h^s=1$ and so every subgroup of $\dis(Q)$ which is normal in $\lmlt(Q)$ is in $\N(Q)$.
\end{remark}

\subsection*{$2$-divisible semimedial left quasigroups} 

The class of $2$-divisible semimedial left quasigroups contains the class of racks.
\begin{lemma}
Racks are $2$-divisible semimedial left quasigroups.
\end{lemma}

\begin{proof}
Let $Q$ be a rack and $a\in Q$. Then $L_{L_a^{\pm 1} (a)}=L_a ^{\pm 1} L_a L_a^{\mp 1}=L_a$ and so $a=a*(a\ldiv a)=(a\ldiv a)*(a\ldiv a)=\s(a\ldiv a)$, and so $\s$ is surjective. If $\s(a)=\s(b)$ then $L_a=L_{a*a}=L_{b*b}=L_b$ and so according to Lemma \ref{s and lambda} $a=b$. Therefore $\s$ is injective.
\end{proof}


If $Q$ is a $2$-divisible semimedial left quasigroup then $\s$ is an automorphism of $Q$ and the lattice $\N(Q)$ is given by 
$$\N(Q)=\setof{N\trianglelefteq \lmlt(Q)}{\s N \s^{-1}\leq N}.$$
Note that $\N(Q)$ contains all the characteristic subgroups of $\dis(Q)$, as the center and the elements of the derived and of the lower central series. 

\begin{proposition}\label{quotient}
Let $Q$ be a $2$-divisible semimedial left quasigroup. Then $\zeta_Q=\c{Z(\dis(Q))}\cap \sigma_{\dis(Q)}$ and 
%
%
%
$Q/\zeta_Q$ is $2$-divisible.
\end{proposition}

\begin{proof}
The conjugation by $\s$ is an automorphism of $\dis(Q)$ and so $\!{Z(\dis(Q))}\leq Z(\dis(Q))$ and therefore $\c{Z(\dis(Q))}$ is a congruence. 
Hence we can apply Lemma \ref{central cong}, since $Q$ is LT and $\c{Z(\dis(Q))}$ is a congruence.

For every $a\in Q$ there exists $b\in Q$ such that $a=\s(b)$. So $[a]=[\s(b)]=\s([b])$ and the map $\s$ over $Q/\zeta_Q$ is surjective. Assume that $[a*a]_{\zeta_Q}=[b*b]_{\zeta_Q}$. Then $L_{a*a} L_{b*b}^{-1}=\s L_a L_b^{-1}\s^{-1}\in Z(\dis(Q))$ and so also $L_a L_b^{-1}\in Z(\dis(Q))$. Moreover, $\s \dis(Q)_a \s^{-1}=\dis(Q)_{\s(a)}=\dis(Q)_{\s(b)}=\s \dis(Q)_b \s^{-1}$ and so $\dis(Q)_a=\dis(Q)_b$. Hence $a\, \zeta_Q \, b$ and $\s$ is injective over $Q/\zeta_Q$.
\end{proof}

Let us define the category $\mathcal{Q}$. The objects of $\mathcal{Q}$ are pairs $((Q,\cdot),f)$ where $(Q,\cdot)$ is a quandle and $f\in \aut{Q,\cdot}$. A morphism $h:(Q_1,f_1)\longrightarrow (Q_2,f_2)$ is a quandle homomorphism between $Q_1$ and $Q_2$ such that $h f_1=f_2 h$, i.e. the morphisms are the commuting square as in Figure \ref{square}. 
\begin{figure}[ht!]

\begin{tikzpicture}[every node/.style={midway}]
  \matrix[column sep={4em,between origins}, row sep={2em}] at (0,0) {
    \node(R) {$Q_1$}  ; & \node(S) {$Q_2$}; \\
    \node(R/I) {$Q_1$}; & \node (T) {$Q_2$};\\
  };
  \draw[<-] (R/I) -- (R) node[anchor=east]  {$f_1$};
  \draw[->] (R) -- (S) node[anchor=south] {$h$};
  \draw[->] (S) -- (T) node[anchor=west] {$f_2$};
 \draw[->] (R/I) -- (T) node[anchor=north] {$h$};
\end{tikzpicture}
\caption{Morphisms in $\mathcal{Q}$.}\label{square}
\end{figure}

%
%

\begin{proposition}\label{equivalence}
Let $\mathcal{S}_2$ the category of $2$-divisible semimedial left quasigroup. The assignment
\begin{eqnarray*}
\mathcal{Q}&\longleftrightarrow &\mathcal{S}_2\notag \\
((Q,\cdot),f) & \mapsto & (Q,*_f), \quad  x*_f y = f(x\cdot y)\label{semimedial formula}\\
 x\cdot_{\s} y =  \s^{-1}(x*y),\quad ((S,\cdot_{\s}),\s) & \leftmapsto & (S,*)\label{quandle op}
\end{eqnarray*}
%
 defines an isomorphism of categories which preserves the displacement groups. In particular, $(Q,*_f)$ is a rack if and only if $f\in C_{\aut{Q,\cdot}}(\lmlt(Q,\cdot))$.

\end{proposition}
\begin{proof}
According to \cite[Lemma 6.66]{book_quasi}, $(Q,*_{\s})$ is a quandle and $\s$ is an automorphism of $(Q,*_{\s})$. On the other hand
\begin{eqnarray*}
(x*_f x)*(y*_f z)&=& f(f(x\cdot x)\cdot f(y\cdot z))\\
&=& f(f(x)\cdot (f(y)\cdot f(z)))\\
&=&f((f(x)\cdot f(y))\cdot (f(x)\cdot f(z)))\\
&=&f((f(x\cdot y))\cdot (f(x\cdot z)))\\
&=&(x*_f y)*_f (x*_f z)
\end{eqnarray*}
and then $(Q,*_f)$ is semimedial. Moreover, $x*_f x=f(x)$ and so $(Q,*_f)$ is $2$-divisible and the two mappings are inverse of each other. 

Let $(Q,*)$ be a semimedial left quasigroup. Let us denote by $\mathcal{L}_a$ the mapping $b\mapsto a*_f b$. Since $L_{a}^{-1} L_b=\mathcal{L}_a^{-1} \s \s^{-1} \mathcal{L}_a=\mathcal{L}_a^{-1} \mathcal{L}_b$, we have that $\dis(Q,*)=\dis(Q,*_f)$. Hence, according to Lemma \ref{on dis_alpha}(iii) $\dis(Q,\cdot)=\dis(Q,*_f)$ (resp. $\dis(Q,*)=\dis(Q,\cdot_\s)$) for every $((Q,\cdot),f)\in \mathcal{Q}$ (resp. $(Q,*)\in \mathcal{S}_2$). So, the isomorphism preserves the displacement groups. 

If $h:((Q_1,\ast_1),f_1)\longrightarrow ((Q_2,\ast_2),f_2)$ is a morphism in $\mathcal{Q}$ then $$h(x *_{f_1} y)=h(f_1(x\cdot_1 y))=f_2 h(x\cdot_1 y)=f_2(h(x)\cdot_2 h(y))=h(x)\ast_{f_2} h(y)$$ and so $h$ is a morphism in $\mathcal{S}_2$ between $(Q_1,*_{f_1})$ and $(Q_2,*_{f_2})$.

If $h:(S_1,\ast_1)\longrightarrow (S_2,\ast_2)$ is a morphism in $\mathcal{S}_2$ then
\begin{eqnarray*}
h(x \ast_1 x) &=& h \s_1(x)= h(x)\ast_{2} h(x)=\s_2 h(x)\\
h(x\cdot_{\s_1} y)&=&h \s_1^{-1}(x\ast_1 y)
=\s_2^{-1}h(x\ast_1 y)\\
&=& \s_2^{-1}(h(x)\ast_2 h(y))=h(x)\cdot_{\s_2} h(y),
\end{eqnarray*}
i.e. $h$ is a morphism in $\mathcal{Q}$ between $((S_1,\cdot_{\s_1}),\s )$ and $((S_2,\cdot_{\s_2}),\s_2)$.

Let  $(Q,\cdot)$ be a quandle and let $\mathcal{L}_x$ be the mapping $y\mapsto x\cdot y$. Then the left multiplication with respect to $\ast_f$ is $L_x=f\mathcal{L}_x$. Hence $(Q,*_f)$ is a rack if and only if
\begin{eqnarray*}
 f\mathcal{L}_x=L_{x}=L_{x*_f x}= f\mathcal{L}_{x*_f x}=f\mathcal{L}_{f(x)}=f^2\mathcal{L}_{x}f^{-1}
\end{eqnarray*}
for every $x\in Q$, i.e. $f\in C_{\aut{Q,\cdot}}(\lmlt(Q,\cdot))$.
\end{proof}
A similar correspondence was observed for racks in \cite{David}. According to Proposition \ref{equivalence}, isomorphism classes of $2$-divisible semimedial left quasigroups with associated quandle $Q$ corresponds to conjugacy classes of automorphisms in $\aut{Q}$.

The correspondence in Proposition \ref{equivalence} does not preserve connectedness: e.g. if $(Q,*)$ is a connected permutation racks then $(Q,\cdot_\s)$ is a projection quandle.

  
 
\begin{proposition}\label{semimedial 2 div nil}
A $2$-divisible semimedial left quasigroup $(Q,*)$ is nilpotent of length $n$ if and only if $(Q,\cdot_\s)$ is nilpotent of length $n$. 
\end{proposition}
 
\begin{proof}
 Since the center of a $2$-divisible semimedial left quasigroup $Q$ is $\zeta_{Q}=\c{Z(\dis(Q))}\cap \sigma_Q$, it is completely determined by the properties of $\dis(Q)$. Hence $\zeta_{(Q,*)}=\zeta_{(Q,\cdot_\s)}$ and the factor $Q/\zeta_{(Q,*)}$ is again $2$-divisible according to Proposition \ref{quotient}. 

Therefore, ascending central series of centers of $(Q,*)$ and of $(Q,\cdot_\s)$ coincide.
\end{proof}

%

\begin{proposition}\label{finite solvable}
A finite $2$-divisible semimedial left quasigroup $(Q,*)$ is solvable if and only if $(Q,\cdot_\s)$ is solvable.
\end{proposition}

\begin{proof}
It follows since $\dis(Q,*)=\dis(Q,\cdot_\s)$.
\end{proof}

Using Lemma \ref{solvable then dis solv} we have that if a $2$-divisible semimedial left quasigroup $(Q,*)$ (resp. a quandle $(Q,\cdot)$ and $f\in \aut{Q,\cdot}$) is solvable of length $n$, then $\dis(Q,*)=\dis(Q,\cdot_\s)$ (resp. $\dis(Q,\cdot)=\dis(Q,*_f)$) is solvable of length at most $2n-1$ and so $(Q,\cdot_\s)$ (resp. $(Q,*_f)$) is solvable of length at most $2n$.

In particular, finite latin quandles are solvable \cite[Corollary 1.3]{CP}. The isomorphism in Theorem \ref{equivalence}  preserve the property of being a quasigroup and then using Proposition \ref{finite solvable} we have the following.

\begin{corollary}\label{semimedial quasi are solv}
Finite $2$-divisible semimedial quasigroups are solvable. 
\end{corollary}

\section{Medial left quasigroup}\label{Sec 4}
The medial law \eqref{M} is actually equivalent to the following equality for the left multiplication mappings:
\begin{equation}\label{mediality_on_mapping}
L_{x*y}L_z=L_{x*z}L_y.
\end{equation}

For medial left quasigroups the converse of Lemma \ref{injective right} holds.
\begin{proposition}\label{faithful medial}
Let $Q$ be a medial left quasigroup. The following are equivalent:
\begin{itemize}
\item[(i)] $Q$ is faithful.
\item[(ii)] The right multiplications of $Q$ are injective. 
\end{itemize}
In particular, finite faithful medial left-quasigroups are quasigroups. 
\end{proposition}


\begin{proof}
The implication (ii) $\Rightarrow$ (i) follows by Lemma \ref{injective right}. Assume that $a*x=b*x$ for some $a,x,y\in Q$. According to \eqref{mediality_on_mapping} we have
$$L_{a*x}L_z=L_{a*z} L_x=L_{b*x}L_z=L_{b*z} L_x$$
for every $z\in Q$. Hence, $L_{a*z}=L_{b*z}$ for every $z\in Q$ and since $Q$ is faithful then $a=b$.
\end{proof}

 Proposition \ref{faithful medial} does not hold for semimedial left quasigroups, indeed there exist finite faithful non-latin quandles.
%
Moreover it shows that finite simple medial left quasigroups of size bigger than $2$, are either quasigroups, if they are faithful, or permutation and isomorphic to $(\mathbb{Z}_p,+1)$, where $p$ is a prime. Examples of infinite faithful medial left-quasigroups that are not quasigroups exist, e.g. $Q=\aff(\mathbb{Z},2,-1,0)$ (note that $Q$ has injective right multiplications but it is not even connected).

Let us show a characterization of medial left quasigroups among the semimedial ones.
\begin{lemma}\label{center}
Let $Q$ be a semimedial left quasigroup. Then 
\[ \c{Z(\dis(Q))}=\{(a,b):\ (x*a)*(b*y)=(x*b)*(a*y)\text{ for every }x,y\in Q\}.\]
\end{lemma}


\begin{proof}
Let $a,b\in Q$. Then, by virtue of \eqref{semi for mappings}, $(x*a)*(b*y)=(x*b)*(a*y)$ holds for every $x,y\in Q$ if and only if
\begin{equation}\label{claim2}
L_x^{-1} L_a L_b^{-1}L_x=L_b^{-1}L_a
\end{equation}
holds for every $x\in Q$. The subgroup $Z(\dis(Q))$ is normal in $\lmlt(Q)$ and so $L_a L_b^{-1}\in Z(\dis(Q))$ if and only if $L_b^{-1} L_a\in Z(\dis(Q))$.

 $(\supseteq)$ If \eqref{claim2} holds, then
$$L_x^{-1} L_y L_b^{-1} L_a L_y^{-1} L_x=L_b^{-1} L_a$$
holds for every $x,y\in Q$, i.e. $L_b^{-1} L_a\in Z(\dis(Q))$. 

$(\subseteq)$ If $L_b^{-1} L_a\in Z(\dis(Q))$ then
$$L_x^{-1} L_y L_b^{-1} L_a L_y^{-1} L_x=L_b^{-1} L_a$$
for every $x,y\in Q$. Setting $y=b$ then \ref{claim2} follows. 
\end{proof}

\begin{corollary}\label{medial iff}
Let $Q$ be a left quasigroup. The following are equivalent:
\begin{itemize}
\item[(i)] $Q$ is medial.
\item[(ii)] $Q$ is semimedial and $\dis(Q)$ is abelian.
\end{itemize}
In particular, abelian semimedial left quasigroups are medial.
\end{corollary}

%
%
%
%
%

Finite simple $2$-divisible semimedial quasigroup are abelian, since they are solvable. Hence according to Corollary \ref{medial iff} they are medial.

%

%
%


%
%
%


Using Lemma \ref{solvable then dis solv}, Proposition \ref{terms for medial} and Corollary \ref{medial iff} together we have the following immediate consequence, which extends \cite[Proposition 5.13]{CP} from medial racks to medial left quasigroups.

\begin{corollary} \label{medial 2 div are nilp}
Medial left-quasigroups are nilpotent of length at most $2$.
\end{corollary}

%
%
%
%

\begin{corollary}\label{unipotent nedial are abelian}
Multipotent medial left quasigroup are connected and abelian.
\end{corollary}

\begin{proof}
Let $Q$ be a multipotent medial left quasigroup. Then the factor $R=Q/\O_{\dis(Q)}$ is multipotent and permutation. According to Proposition \ref{multipotent have right mult 1-1}, $R$ is faithful and so it is trivial. Then $\dis(Q)$ is abelian and transitive on $Q$. Thus $Q$ is connected and $\dis(Q)$ is regular. Since $Q$ is LT, then $Q$ is abelian. 
\end{proof}

 \begin{remark}
The equivalence of Proposition \ref{equivalence} restricts and corestricts to the subcategory of medial quandles and of medial $2$-divisible left quasigroups: indeed $\dis(Q,*)=\dis(Q,\cdot_\s)$ and so, according to Corollary \ref{medial iff}, $(Q,*)$ is medial if and only if $(Q,\cdot_\s)$ is medial. 

The structure of medial quandle and their homomorphisms has been described in \cite{Medial} and \cite{Bonatto2018OnTS} using a construction involving abelian groups and their homomorphisms. Hence the structure of $2$-divisible medial left quasigroups can be obtained directly from the one of medial quandles. 
%
\end{remark}

%

The left quasigroup $Q\times_\theta S$ defined in \eqref{extensions by theta} is medial if and only if $Q$ is medial and
	\begin{equation}\label{mediality for cocycles}
\theta_{a*b,c*d}\theta_{c,d}=
\theta_{a*c,b*d}\theta_{b,d}\tag{MC}
\end{equation}
for every $a,b,c,d\in Q$. If \eqref{mediality for cocycles} holds we say that $\theta$ is a {\it medial cocycle} of $Q$.  Let $Q$ be a medial quasigroup and $\gamma:Q\longrightarrow \sym_Q$ be a mapping. Then
\begin{eqnarray}\label{normalized cocycle}
	\widetilde{\theta}_{a,b}&=&\gamma_{a*b}\theta_{a,b}\gamma_{b}^{-1}
\end{eqnarray}
is a medial cocycle whenever $\theta$ is a medial cocycle and $Q\times_{\widetilde{\theta}} S$ is isomorphic to $Q\times_\theta S$. If $Q$ is a quasigroup and $u\in Q$, following \cite{MeAndPetr}, we can define $\gamma_a=\theta_{a/u,u}^{-1}$ and so the cocycle defined in \eqref{normalized cocycle} has the property that $\widetilde{\theta}_{a,u}=\theta_{u/u,u}$ for every $a\in Q$. We say that $\widetilde{\theta}$ is a {\it $u$-normalized cocycle}.
\begin{proposition}\label{normalized cocycles are constant}
Let $Q$ be a medial left-quasigroup. If $Q/\lambda_Q$ is a quasigroup then $Q$ is isomorphic to the direct product of $Q/\lambda_Q$ and a permutation left-quasigroup.
\end{proposition}

\begin{proof}
Since $Q/\lambda_Q$ is connected, then $\lambda_Q$ is uniform and therefore $Q\cong Q/\lambda_Q\times_\theta S$. Let $u\in Q$ and consider the $u$-normalized cocycle defined as in \eqref{normalized cocycle} using  $\gamma_{a}=\theta_{a/u,u}^{-1}$. 
Then \eqref{mediality for cocycles} implies that
\begin{eqnarray}
\widetilde{\theta}_{a*b,c*u}&=&\widetilde{\theta}_{a*c,b*u} \label{NC1}\\
\widetilde{\theta}_{a,u}&=&\widetilde{\theta}_{u/u,u}\label{NC2}
\end{eqnarray}
 for every $a,b,c\in Q$. Setting $c=u/u$ in \eqref{NC1} we have that $
\widetilde{\theta}_{a*b,u}=
\widetilde{\theta}_{a*(u/u),b*u}$, and according to \eqref{NC2} $
\widetilde{\theta}_{a*(u/u),b*u}=
\widetilde{\theta}_{u,u/u}$ for every $a,b\in Q$. Since $Q$ is a quasigroup then $
\widetilde{\theta}_{a,b}=
\widetilde{\theta}_{u,u/u}$ for every $a,b\in Q$. Then $Q\times_\theta S\cong Q\times_{\widetilde{\theta}} S$ which is the direct product of $Q$ and the permutation left quasigroup $(S,\widetilde{\theta}_{u,u/u})$.
\end{proof}

\subsection*{Medial racks}



According to Corollary \ref{medial iff}, a rack $Q$ is medial if and only if $\dis(Q)$ is abelian (such characterization can be found also in \cite{Medial}).
Examples of medial racks are the following:
\begin{itemize}
\item[(i)] 
any permutation left quasigroup is a medial rack.
\item[(ii)] The affine left-quasigroup $\aff(A,1-f,f,c)$ where $f(c)=c$ is a medial rack.

\end{itemize}

Let $Q$ be a left quasigroup, we define the equivalence relation $\m_Q$ as
\begin{eqnarray*}
a\, \m_Q\, b\,  \text{ if and only if } \,Sg(a)=Sg(b).
\end{eqnarray*}
 The equivalence $\m_Q$ is actually a congruence for racks \cite[Proposition 7.1]{covering_paper} and $\m_Q\leq\lambda_Q$.

\begin{theorem}\label{medial superconnected racks}
Let $Q$ be a superconnected medial rack. Then 
$$Q\cong \aff(A,1-f,f,0)\times (C,+1)
$$ 
where $C$ is a cyclic group, $A$ is a abelian group and $1-f\in\aut{A}$. 
%
\end{theorem}

\begin{proof} 
%
 The quandle $Q/\m_Q$ is a superconnected medial quandle and so it is latin and $Q/\m_Q\cong \aff(A,1-f,f,0)$ where $1-f\in \aut{A}$ \cite[Corollary 2.6]{Principal}. Then $Q\cong Q/\m_Q\times_\theta S$ for some medial cocycle $\theta$.  Hence, according to Proposition \ref{normalized cocycles are constant}, $Q$ is the direct product of $Q/\m_Q$ and $(S,\theta_{0,0})$ which is a connected permutation rack, and so it is isomorphic to $(C,+1)$ for some cyclic group $C$.
%
%
\end{proof}

%
%
%
%
We can extend the results on abelian quandles of \cite{affine_quandles} to medial superconnected rack.
\begin{corollary}\label{abelian iff aff for racks}
Let $Q$ be a superconnected rack. The following are equivalent:
\begin{itemize}
\item[(i)] $Q$ is medial.
\item[(ii)] $Q$ is affine.
\item[(iii)] $Q$ is abelian.
\end{itemize}
\end{corollary}
\begin{proof}
(i) $\Rightarrow$ (ii) According to Theorem \ref{medial superconnected racks}, $Q\cong \aff(A,1-f,f,0)\times (C,+1)$. The rack $Q= (C,+1)$ is isomorphic to $\aff(C,0,id,1)$. The direct product of affine quandles is affine, and in particular we have that $Q\cong \aff(A\times C,(1-f, 0), (f, id), (0,1))$.

(ii) $\Rightarrow$ (iii) Affine left quasigroup are abelian.

(iii) $\Rightarrow$ (i) If $Q$ is abelian, then $\dis(Q)$ is abelian and so $Q$ is medial.
\end{proof}

In particular, Theorem \ref{medial superconnected racks} and Corollary \ref{abelian iff aff for racks} applies to finite medial connected racks.

\begin{corollary}\label{finite medial connected racks}
Let $Q$ be a finite connected rack. The following are equivalent:
\begin{itemize}
\item[(i)] $Q$ is medial.
\item[(ii)] $Q$ is abelian.
\item[(iii)] $Q\cong \aff(A,1-f,f,0)\times(\mathbb{Z}_n,+1)$ where $A$ is a abelian group and $Fix(f)=0$.
\end{itemize}
\end{corollary}

\begin{proof}
Assume that $Q$ is medial. Since $Q$ is finite, then $Q/\m_Q$ is a finite connected medial quandle. Then $Q/\m_Q$ is latin and then superconnected. According to \cite[Proposition 3.4]{Maltsev_paper} then $Q$ is superconnected.
According to Corollary \ref{abelian iff aff for racks} (ii) holds and and by Theorem \ref{medial superconnected racks} (iii) holds. On the other hand, if (iii) holds, then $Q$ is medial since it is the direct product of medial racks.
\end{proof}

\section{Mal'tsev conditions}\label{Sec 5}

Mal'tsev conditions for left quasigroups have been investigated in \cite{Maltsev_paper}, with particular focus on quandles. In this section we partially extend the same results to medial left quasigroups. 

In \cite{Maltsev_paper} we proved that if a variety of left quasigroups satisfies a non-trivial Malt'sev condition then it has a {\it Malt'sev term} term, i.e.  a ternary term $m$ satisfying the following identities:
$$m(x,y,y)\approx x \approx m(y,y,x).$$
A variety is said to be {\it Malt'sev} if it has a Malt'sev term.

The variety of $n$-multipotent left quasigroups is axiomatized by the identity
$$L_{\s^{n-1}(x)}L_{\s^{n-2}(x)}\ldots L_{\s(x)}L_x(x)\approx L_{\s^{n-1}(y)}L_{\s^{n-2}(y)}\ldots L_{\s(y)}L_y(y).$$
According to \cite[Section 2]{Maltsev_paper} every variety of multipotent left quasigroup has a Malt'sev term. In particular, every multipotent left quasigroup is superconnected. The term
$$m(x,y,z)=\left(L_{\s^{n-2}(x)}\ldots L_{\s(x)}L_x\right)^{-1}  L_{\s^{n-2}(y)}\ldots L_{\s(y)}L_y(z)$$
is a Malt'sev term for $n$-multipotent left quasigroups. 

Any Malt'sev variety omits strongly abelian congruences \cite[Theorem 3.13]{shape}. In particular, the left quasigroups in a Cayley Malt'sev variety are faithful. Let us start with a general observation, following by this fact. 

\begin{proposition}\label{dis is transitive}
Let $\mathcal{V}$ be a Malt'sev variety of left quasigroups. Then $\dis(Q)$ is transitive on $Q$ for every $Q\in \mathcal{V}$.
\end{proposition}

\begin{proof}
Let $Q\in \mathcal{V}$ and $P=Q/\mathcal{O}_{\dis(Q)}$. According to Corollary \ref{on orbits2}, $P$ is a permutation left quasigroup and according to \cite[Theorem 5.3]{covering_paper} $1_P=\lambda_{P}$ is a strongly abelian congruences. The variety $\mathcal{V}$ does not contain any non-trivial strongly abelian congruence, then $1_P=0_P$, i.e. $\dis(Q)$ is transitive on $Q$.
\end{proof}

Finite Malt'sev medial left quasigroups are quasigroups.
\begin{lemma}
Let $Q$ be a finite medial left-quasigroup. The following are equivalent:
\begin{itemize}
\item[(i)] The variety generated by $Q$ is Mal'tsev.
\item[(ii)] $Q$ is a medial quasigroup.
\end{itemize}
\end{lemma}

\begin{proof}
In the variety generated by $Q$ there is no strongly abelian congruence. Then $Q$ is faithful and so $Q$ is a medial quasigroup by virtue of  Proposition \ref{faithful medial}.
\end{proof}

\begin{question}
Do infinite medial Malt'sev left quasigroups which are not quasigroups exist?
\end{question}

Meet-semidistributive varieties omit solvable congruences \cite[Theorem 8.1]{shape} and then in particular every congruence is equal to the commutator with itself. Therefore we have the following.

\begin{proposition}\label{no medial in meet semi}
Let $\mathcal{V}$ be a meet-semidistributive variety of semimedial left quasigroup. Then:
\begin{itemize}
\item[(i)] $\mathcal{V}$ does not contain any non-trivial medial left quasigroup.
\item[(ii)] $\mathcal{V}$ does not contain any non-trivial finite $2$-divisible quasigroup.

\item[(iii)] $E(Q)$ is either infinite or trivial for every $Q\in \mathcal{V}$.
\end{itemize}
\end{proposition}

%
%

\begin{proof}
According to Corollary \ref{medial 2 div are nilp} medial left quasigroups are nilpotent, and according to Proposition \ref{semimedial quasi are solv}, finite $2$-divisible semimedial left quasigroups are solvable and so $\mathcal{V}$ does not contain any non-trivial medial left quasigroup or finite $2$-divisible quasigroup. Let $Q\in \mathcal{V}$. Then the subalgebra $E(Q)$ is a meet-semidistributive quandle. According to \cite[Theorem 4.6]{Maltsev_paper}, there is no non-trivial finite meet-semidistributive quandle and so $E(Q)$ is either infinite or trivial. 
%
%
%
\end{proof}

\begin{corollary}
There is no meet-semidistributive variety of medial left-quasigroups. 
\end{corollary}

\begin{corollary}
Let $\mathcal{V}$ be a meet-semidistributive variety of semimedial left quasigroup and $Q\in\mathcal{V}$. Then: 
\begin{itemize}
\item[(i)] $\dis_\alpha$ is perfect for every $\alpha\in Con(Q)$.
\item[(ii)] The only solvable subgroup in $\N(Q)$ is the trivial subgroup.
\item[(iii)] If $Q$ is $2$-divisible then $Z(N)=1$ for every $N\in \N(Q)$.
\end{itemize}
\end{corollary} 

\begin{proof}
(i) The left quasigroups in $\mathcal{V}$ are faithful, since $\mathcal{V}$ is a Cayley variety. Let $Q\in \mathcal{V}$ and $\alpha\in Con(Q)$. According to Proposition \ref{p:comm_comm} and since $\alpha=[\alpha,\alpha]$ we have that $\dis_\alpha=\dis_{[\alpha,\alpha]}\leq [\dis_\alpha,\dis_\alpha]\leq \dis_\alpha$ for every $\alpha\in Con(Q)$ and every $Q\in\mathcal{V}$.

%
%
%

(ii) If $N\in \N(Q)$ is solvable of length $n$ and let $D$ be the non-trivial $(n-1)$th element of the derived series of $N$. So $D$ is abelian and it is in $\N(Q)$. Hence, according to Lemma \ref{abelian subgroup gives abelian cong for LTT}, $\beta=\mathcal{O}_{D}$ is abelian and therefore trivial. Hence $D=1$, contradiction.

(iii) If $Q$ is $2$-divisible and $N\in\N(Q)$, then $Z(N)$ is a characteristic subgroup of $N$, and so it is normal in $\lmlt(Q)$ and $\s Z(N) \s^{-1}\leq Z(N)$. So $Z(N)  \in \N(Q)$ and so we can conclude using (ii).
\end{proof}

%

\section{The spelling property}\label{Sec 6}
%

We say that a left-quasigroup $Q$ has the {\it spelling property} ($Q$ is a {\it spelling left-quasigroup}) if 
there exists two binary terms $w^{\pm}$ in the language of groups such that
$$ L_{L_x^{\pm 1} (y)}=w^{\pm} (L_x,L_y).$$

The spelling property is equivalent to a pair of identities of the form 
\begin{equation}\label{w_pm}
L_x^{\pm 1} (y)*z\approx w^{\pm} (L_x,L_y)(z).
\end{equation} 

The identities \eqref{w_pm} hold for any subalgebra, factor and power of $Q$ and so they also have the spelling property.

\begin{remark*}
The class of spelling left-quasigroups is stable under taking subalgebras, homomorphic images and powers.\end{remark*}

Let us show some examples of left quasigroups with the spelling property:
\begin{itemize}

\item[(i)] Racks have the spelling property. Indeed \eqref{LD} is equivalent to
$$
 L_{L_x^{\pm 1}(y) }=L_{x}^{\pm 1} L_y L_x^{\mp 1}.$$ 

\item[(ii)] Let $Q$ be a left quasigroup such that $\mathcal{O}_{\lmlt(Q)}\leq \lambda_Q$. Then 
$$L_{x*y}=L_{x\ldiv y}=L_y $$ holds for every $x,y\in Q$ and so $Q$ has the spelling property.

\item[(iii)] A left quasigroup $Q$ is {\it associative} if
\begin{align*}\label{associativity}
(x*y)*z \approx x*(y*z)\,  \Leftrightarrow  \, L_{x*y}=L_x L_y 
\end{align*}
holds.
 The associative law 
 implies that
\begin{eqnarray*}
	L_y= L_{x* (x\ldiv y)}= L_x L_{x\ldiv y} \, &\Leftrightarrow \,& L_{x\ldiv y}=L_x^{-1} L_y
\end{eqnarray*}
for every $x,y\in Q$.  Associative left quasigroups are right-groups (left cancellative, right simple semigroups) and then according to \cite[Exercise 2.6.6]{semigroups} they are the direct product of a group and a projection left quasigroup.
\end{itemize}

\begin{lemma}
	Let $Q$ be a spelling left quasigroup and $S\subseteq Q$. Then
	$$\bigcap_{a\in S} \setof{b\in Q}{b*a=a}$$ is a subalgebra of $Q$.
\end{lemma}
\begin{proof}
It is enough to prove that $\setof{b\in Q}{b*a=a}$ is a subalgebra for every $a\in Q$. 	If $b*a=L_b(a)=c*a=L_c(a)=a$ then $L_{b*c}$ and $L_{b\ldiv c}$ belongs to the subgroups generated by $L_b, L_c$, which acts trivially on $a$. Therefore $b*c$ and $b\ldiv c$ fix $a$.
\end{proof}

Note that every term $t$ in the language on left quasigrups is given by
\begin{equation}\label{a general term}
t(x_1,\ldots,x_{m})=p(x_{i_1},\ldots, x_{i_k})\bullet q(x_{j_1},\ldots,x_{j_s})
\end{equation}
 for suitable subterms $p,q$ and $\bullet\in \{\ast,\backslash\}$ and $x_{i_l},x_{j_l}\in \{x_1,\ldots,x_{n}\}$.
%


\begin{lemma}\label{on L_t}
Let $Q$ be a spelling left-quasigroup and $t$ be a term. Then
\begin{equation}\label{eq for L_t}
L_{t(x_1,\ldots,x_n)}=w_t(\setof{L_{x_i}}{1\leq i \leq n})
\end{equation}
where $w_t$ is a word which depends only on the term $t$. In particular, if $b\in Sg(a_1,\ldots,a_n)$ then $L_b\in \langle \setof{L_{a_i}}{1\leq i\leq n}\rangle$.
\end{lemma}
\begin{proof}
 Let $t$ be an $m$-ary term. We proceed by induction on the number of occurrences of variables in $t$. If the number of occurrences is $1$ we are done. Let $t$ be a term as in \eqref{a general term} with $n$ occurrences of variables. By induction, we have that $L_{p(x_{i_1},\ldots, x_{i_k})}=w_p(\setof{L_{x_{i_l}} }{1\leq l\leq k})$ and $L_{q(x_{j_1},\ldots, x_{j_s})}=w_q(\setof{L_{x_{j_l}} }{1\leq l\leq s})$ for suitable words which depends only on the terms $p$ and $q$. Therefore $L_{t}=w^{\pm}(w_p(\setof{L_{x_{i_l}} }{1\leq l\leq k}),w_q(\setof{L_{x_{j_l}} }{1\leq l\leq s}))$ and so it is a word in $\setof{L_{x_i} }{1\leq i \leq m}$ which depends just on $t$.
\end{proof}

According to Lemma \ref{on L_t}, spelling left quasigroups respect the condition described in Remark \ref{like spelling}, and therefore, they have the LT property.

\begin{corollary}\label{sufficient condition to be LTT}
Spelling left-quasigroups are LT left quasigroups. In particular, every term $t$ is equivalent to one of the form
\begin{equation}\label{term for spelling lQGs}
t(x_1,\ldots ,x_n)\approx L_{x_{i_1}}^{k_1}\ldots L_{x_{i_m}}^{k_m}(x_{i_{m+1}}),
\end{equation}
where $x_{i_j}\in \{x_1,\ldots ,x_n\}$ for every $1\leq j\leq m+1$.
\end{corollary}

\begin{proof}

Let $t$ be a term as in \eqref{a general term}. By induction on the number of occurrences, $q$ is a LT term as in \eqref{term for spelling lQGs} and by \eqref{eq for L_t} in Lemma \ref{on L_t} we have
$$t(x_1,\ldots,x_m)=w_p(\setof{L_{x_{i_l}}}{1\leq l \leq k})^{\pm 1}  (q(x_{j_1},\ldots x_{j_s})).$$
Therefore $t$ is an LT term and it is as in \eqref{term for spelling lQGs}.
%
\end{proof}

\begin{corollary}\label{monog LTT}
Let $Q$ be a spelling left-quasigroup and $S\subseteq Q$. Then 
$$Sg(S)=\bigcup_{a\in S} a^M $$ 
where $M=\langle L_a, \, a\in S\rangle$. In particular, $Sg(a)=\setof{L_a^k (a)}{k\in \mathbb{Z}}$ for every $a\in Q$.
\end{corollary}
\begin{proof}
All the terms of $Q$ are equivalent to LT terms as in \eqref{term for spelling lQGs}. Every element $b\in Sg(S)$ is a given by an LT term evaluated on the elements of $S$. Therefore $b=h(s)$ for some $s\in S$ and some $h\in \langle L_a, a\in S\rangle$.
\end{proof}

The Galois connection defined for racks in \cite[Section 3.3]{CP} depends only on the spelling property.

\begin{theorem}\label{Galois_correspondence}
Let $Q$ be a left quasigroup with the spelling property. The assignment $\alpha\mapsto\dis_\alpha$ and $N\mapsto\c{N}$ is a monotone Galois connection between $Con(Q)$ and the lattice of the normal subgroups of $\lmlt(Q)$.
\end{theorem}

\begin{proof}
First we need to show that $\c{N}$ is a congruence for every $N\trianglelefteq \lmlt(Q)$. Then we can conclude as in Theorem \ref{Galois for semimedial}. Let $\alpha=\c{N}$ and 
let assume that $a_1\,\alpha\, b_1$ and $a_2\,\alpha\,b_2$ and that $L_{a_1*a_2}=L_{a_{i_1}}^{k_1}\ldots L_{a_n}^{k_n}$ where $k_j\in \{\pm 1\}$. Then
\begin{eqnarray}\label{manipulation}
L_{a_1*a_2}L_{b_1*b_2}^{-1} &=& L_{a_{i_1}}^{k_1}\ldots \underbrace{L_{a_{i_n}}^{k_n} L_{b_{i_n}}^{-k_n}}_{\in N}\ldots L_{b_{i_1}}^{-k_1}=\underbrace{L_{a_{i_1}}^{k_1}L_{b_{i_1}}^{-k_1}}_{\in N}L_{b_{i_1}}^{k_1}\ldots \underbrace{L_{a_{i_n}}^{k_n} L_{b_{i_n}}^{-k_n}}_{\in N}\ldots L_{b_{i_1}}^{-k_1}\notag\\
&=& \underbrace{L_{a_{i_1}}^{k_1}L_{b_{i_1}}^{-k_1}}_{\in N}L_{b_{i_1}}^{k_1}L_{a_{i_2}}^{k_2} \ldots \underbrace{L_{a_{i_n}}^{k_n} L_{b_{i_n}}^{-k_n}}_{\in N}\ldots L_{b_{i_2}}^{-k_2} L_{b_{i_1}}^{-k_1}=\notag\\
&=& \underbrace{L_{a_{i_1}}^{k_1}L_{b_{i_1}}^{-k_1}}_{\in N}L_{b_{i_1}}^{k_1} \underbrace{L_{a_{i_2}}^{k_2}L_{b_{i_2}}^{-k_2}}_{\in N}L_{b_{i_2}}^{k_2} \ldots \underbrace{L_{a_{i_n}}^{k_n} L_{b_{i_n}}^{-k_n}}_{\in N}\ldots L_{b_{i_2}}^{-k_2} L_{b_{i_1}}^{-k_1}.
\end{eqnarray}
Iterating the manipulation in \eqref{manipulation} and using that $N$ is a normal subgroup of $\lmlt(Q)$ we get that $L_{a_1*a_2}L_{b_1*b_2}^{-1}\in N$. Similarly we can prove that  $L_{a_1\ldiv a_2}L_{b_1\ldiv b_2}^{-1}\in N$ and therefore $\alpha$ is a congruence.
\end{proof}

According to Corollary \ref{sufficient condition to be LTT}, we can apply Proposition \ref{p:commutators} and Corollary  \ref{ab an central cong iff} to spelling left quasigroups. By virtue of Lemma \ref{central cong} we can characterize the center of spelling left quasigroups.

\begin{corollary}
Let $Q$ be a spelling left quasigroup. Then $\zeta_Q=\c{Z(\dis(Q))}\cap \sigma_{\dis(Q)}$.
\end{corollary}


\bibliographystyle{amsalpha}
\bibliography{references} 

\def\cprime{$'$} \def\cprime{$'$}
\providecommand{\bysame}{\leavevmode\hbox to3em{\hrulefill}\thinspace}
\providecommand{\MR}{\relax\ifhmode\unskip\space\fi MR }
\providecommand{\MRhref}[2]{%
  \href{http://www.ams.org/mathscinet-getitem?mr=#1}{#2}
}
\providecommand{\href}[2]{#2}
\begin{thebibliography}{JPSZD18}

\bibitem[AG03]{AG}
Nicol{\'a}s Andruskiewitsch and Mat{\'{\i}}as Gra{\~n}a, \emph{From racks to
  pointed {H}opf algebras}, Adv. Math. \textbf{178} (2003), no.~2, 177--243.
  \MR{1994219 (2004i:16046)}

\bibitem[BCW19]{Bonatto2018OnTS}
Marco Bonatto, Alissa~S. Crans, and Glen Whitney, \emph{On the structure of
  {H}om quandles}, J. Pure Appl. Algebra \textbf{223} (2019), no.~11,
  5017--5029. \MR{3955053}

\bibitem[Ber12]{UA}
Clifford Bergman, \emph{Universal algebra}, Pure and Applied Mathematics (Boca
  Raton), vol. 301, CRC Press, Boca Raton, FL, 2012, Fundamentals and selected
  topics. \MR{2839398}

\bibitem[BKSV19]{Rumples}
Marco {Bonatto}, Michael {Kinyon}, David {Stanovsk{\'y}}, and Petr
  {Vojt{\v{e}}chovsk{\'y}}, \emph{{Involutive latin solutions of the
  Yang-Baxter equation}}, arXiv e-prints (2019), arXiv:1910.02148.

\bibitem[Bon19]{Maltsev_paper}
Marco Bonatto, \emph{Maltsev classes of left-quasigroups and quandles}, arXiv
  e-prints (2019), arXiv:1904.13388.

\bibitem[Bon20]{Principal}
\bysame, \emph{Principal and doubly homogeneous quandles}, Monatshefte f{\"u}r
  Mathematik \textbf{191} (2020), no.~4, 691--717.

\bibitem[BS]{covering_paper}
Marco Bonatto and David Stanovsk\'{y}, \emph{An universal algebraic approach to
  quandle coverings}, In preparation.

\bibitem[BS19]{CP}
Marco {Bonatto} and David {Stanovsk{\'y}}, \emph{{Commutator theory for racks
  and quandles}}, arXiv e-prints (2019), arXiv:1902.08980.

\bibitem[BV18]{MeAndPetr}
Marco Bonatto and Petr Vojt\v{e}chovsk\'{y}, \emph{Simply connected latin
  quandles}, J. Knot Theory Ramifications \textbf{27} (2018), no.~11, 1843006,
  32. \MR{3868935}

\bibitem[EG14]{Even}
Val\'{e}rian Even and Marino Gran, \emph{On factorization systems for
  surjective quandle homomorphisms}, J. Knot Theory Ramifications \textbf{23}
  (2014), no.~11, 1450060, 15. \MR{3293045}

\bibitem[FM87]{comm}
Ralph Freese and Ralph McKenzie, \emph{Commutator theory for congruence modular
  varieties}, London Mathematical Society Lecture Note Series, vol. 125,
  Cambridge University Press, Cambridge, 1987. \MR{909290}

\bibitem[GI18]{Gateva}
Tatiana Gateva-Ivanova, \emph{Set-theoretic solutions of the {Y}ang-{B}axter
  equation, braces and symmetric groups}, Adv. Math. \textbf{338} (2018),
  649--701. \MR{3861714}

\bibitem[HD88]{TCT}
McKenzie~R. Hobby~D., \emph{The structure of finite algebras}, Contemporary
  Mathematics, vol.~76, American Mathematical Society, 1988.

\bibitem[How95]{semigroups}
John~M. Howie, \emph{Fundamentals of semigroup theory}, London Mathematical
  Society Monographs. New Series, vol.~12, The Clarendon Press, Oxford
  University Press, New York, 1995, Oxford Science Publications. \MR{1455373}

\bibitem[Joy82]{J}
David Joyce, \emph{A classifying invariant of knots, the knot quandle}, J. Pure
  Appl. Algebra \textbf{23} (1982), no.~1, 37--65. \MR{638121 (83m:57007)}

\bibitem[JPSZD15]{Medial}
P\v{r}emysl Jedli\v{c}ka, Agata Pilitowska, David Stanovsk\'{y}, and Anna
  Zamojska-Dzienio, \emph{The structure of medial quandles}, J. Algebra
  \textbf{443} (2015), 300--334. \MR{3400403}

\bibitem[JPSZD18]{affine_quandles}
\bysame, \emph{Subquandles of affine quandles}, J. Algebra \textbf{510} (2018),
  259--288. \MR{3828785}

\bibitem[JPZ19]{Jedlicka}
P{\v{r}}emysl {Jedli{\v{c}}ka}, Agata {Pilitowska}, and Anna
  {Zamojska-Dzienio}, \emph{{Distributive biracks and solutions of the
  Yang-Baxter equation}}, International Journal of Algebra and Computation
  (2019).

\bibitem[KK13]{shape}
Keith~A. Kearnes and Emil~W. Kiss, \emph{The shape of congruence lattices},
  Mem. Amer. Math. Soc. \textbf{222} (2013), no.~1046, viii+169. \MR{3076179}

\bibitem[KP04]{MJD}
Michael~K. Kinyon and J.~D. Phillips, \emph{Axioms for trimedial quasigroups},
  Comment. Math. Univ. Carolin. \textbf{45} (2004), no.~2, 287--294.
  \MR{2075277}

\bibitem[MS05]{MS}
Ralph McKenzie and John Snow, \emph{Congruence modular varieties: commutator
  theory and its uses}, Structural theory of automata, semigroups, and
  universal algebra, NATO Sci. Ser. II Math. Phys. Chem., vol. 207, Springer,
  Dordrecht, 2005, pp.~273--329. \MR{2210134}

\bibitem[PR98]{Modes}
A.~Pilitowska and A.~Romanowska, \emph{Reductive modes}, Period. Math. Hungar.
  \textbf{36} (1998), no.~1, 67--78. \MR{1684506}

\bibitem[Rum05]{Rump}
Wolfgang Rump, \emph{A decomposition theorem for square-free unitary solutions
  of the quantum {Y}ang-{B}axter equation}, Adv. Math. \textbf{193} (2005),
  no.~1, 40--55. \MR{2132760}

\bibitem[Shc17]{book_quasi}
Victor Shcherbacov, \emph{Elements of quasigroup theory and applications},
  Monographs and Research Notes in Mathematics, CRC Press, Boca Raton, FL,
  2017. \MR{3644366}

\bibitem[{Sta}11]{David}
David {Stanovsk{\'y}}, \emph{Selfdistributive groupoids a2: Non-idempotent left
  distributive left quasigroups}, Acta Univ. Carolinae Math. Phys.
  \textbf{52/2} (2011), 7--28.

\bibitem[SV14]{SV1}
David Stanovsk\'{y} and Petr Vojt\v{e}chovsk\'{y}, \emph{Commutator theory for
  loops}, J. Algebra \textbf{399} (2014), 290--322. \MR{3144590}

\end{thebibliography}

\end{document}